\documentclass{amsart}
\usepackage{amsfonts}
\usepackage{amsmath}
\usepackage{amssymb}
\usepackage{amscd}
\usepackage{amstext}
\usepackage{amsthm}
\usepackage{mathrsfs}
\usepackage{enumerate}
\usepackage{color}

\theoremstyle{plain}
\newtheorem{lem}{Lemma}[section]
\newtheorem{prop}[lem]{Proposition}
\newtheorem{thm}[lem]{Theorem}
\newtheorem{cor}[lem]{Corollary}
\newtheorem{fact}[lem]{Fact}

\newtheorem{mainthm}{Theorem}

\theoremstyle{definition}
\newtheorem{defn}[lem]{Definition}

\theoremstyle{remark}

\newtheorem{rem}[lem]{Remark}

\begin{document}
\title[A diameter bound for Sasaki manifolds]
{A diameter bound for Sasaki manifolds \\ with application to uniqueness for Sasaki-Einstein structure}
\author{Yasufumi Nitta \and Ken'ichi Sekiya}
\address{Department of Mathematics, Graduate school of science, Osaka University. 1-1 Machikaneyama, Toyonaka, Osaka 560-0043 Japan.}
\address{Department of Mathematics, Graduate school of science, Osaka University. 1-1 Machikaneyama, Toyonaka, Osaka 560-0043 Japan.}
\date{}
\maketitle
\thispagestyle{empty}

\begin{abstract}
In this paper we give a diameter bound for Sasaki manifolds with positive transverse Ricci curvature. 
As an application, we obtain the uniqueness of
Sasaki-Einstein metrics on compact Sasaki manifolds modulo the action
of the identity component of the automorphism group for the transverse holomorphic structure. 
\end{abstract}

\section{Introduction}
A {\it Sasaki manifold} is a Riemannian manifold $(S, g)$ whose cone metric $\bar g = dr^{2} + r^{2}g$ on $C(S) = S \times \mathbb{R}_{+}$ is K${\rm \ddot a}$hler. 
Then Sasakian geometry sits naturally in two aspects of K${\rm \ddot a}$hler geometry, since for one thing, $(S, g)$ is the base of the {\bf cone manifold} $(C(S), \bar g)$ which is K${\rm \ddot a}$hler, and for another thing any Sasaki manifold is contact, and the one dimensional foliation associated to the characteristic Reeb vector field admits a {\bf transverse K${\rm \bf \ddot a}$hler structure}.

The main purpose of this paper is to prove a Myers' type theorem for Sasaki manifolds and give a diameter bound for complete Sasaki manifolds with positive transverse Ricci curvature. 
Our main result is stated as follows. 

\begin{mainthm}\label{Myers2}
Let $(S, g)$ be a $(2n + 1)$ dimensional complete Sasaki manifold with Sasakian structure $\mathcal{S} = \{g, \xi, \eta, \Phi\}$. 
Suppose $Ric^{T} \geq \tau g^{T}$ for some constant $\tau > 0$. 
Then 
\begin{equation*}
{\rm diam}(S, g) \leq 2\pi \sqrt{\frac{2n - 1}{\tau}}. 
\end{equation*}
\end{mainthm}

As an application of Theorem A, we have uniqueness of Sasaki-Einstein metrics up to the action of the identity component of the automorphism group for the transverse holomorphic structure. 
For toric cases, the uniqueness of Sasaki-Einstein metrics was recently obtained by Cho, Futaki and Ono \cite{CFO} by showing that the argument of Guan \cite{G} is valid also for the space of K${\rm \ddot a}$hler potentials for the transverse K${\rm \ddot a}$hler structure. 

In this paper, we shall prove such uniqueness without toric assumption by applying Theorem A and the argument of Bando and Mabuchi in \cite{BM}. 
\begin{mainthm}\label{Maintheorem}
Let $(S, g)$ be a compact Sasaki manifold with Sasakian structure $\mathcal{S} = \{g, \xi, \eta, \Phi\}$. 
Assume that the set $\mathscr{E}$ of all Sasaki-Einstein metrics which is compatible with $g$ is non-empty. 
Then the identity component of the automorphism group for the transverse holomorphic structure acts transitively on $\mathscr{E}$. 
\end{mainthm}

This paper is organized as follows: 
In Section 2, we give a brief review of Sasakian geometry and transverse K${\rm \ddot a}$hler geometry. 
In Section 3, by showing a Myers' type theorem on complete Sasaki manifolds, we give a proof of Theorem \ref{Myers2}. 
Our poof is then based on a variational formula for a minimizing normal geodesic in the sense of sub-Riemannian geometry (see \cite{M} for example). 
Finally in Section 4, we shall show that an argument similar to Bando and Mabuchi \cite{BM} allows us to obtain a proof of Theorem B. 

The first-named author would like to express his gratitude to Professor T. Mabuchi for valuable comments. 
The second-named author would like to thank Professor R. Goto for helphul advice. 
\section{Brief review of Sasakian geometry}
\subsection{Sasaki manifolds}
We recall the basic theory of Sasaki manifolds. 
For the details, see \cite{BG} and \cite{FOW}. 
Throughout this paper, we assume that all manifolds are connected. 
Let $(S, g)$ be a Riemannian manifold and $(C(S), \bar{g}) = (S \times \mathbb{R}_{+}, dr^{2} + r^{2}g)$ be its cone manifold, where $\mathbb{R}_{+} = \{ x \in \mathbb{R}\ |\ x >0 \}$ and $r$ is the standard coordinate on $\mathbb{R}_{+}$. 
\begin{defn}
$(S, g)$ is called a Sasaki manifold if the cone manifold $(C(S), \bar{g})$ is a K${\rm \ddot a}$hler manifold. 
\end{defn}
A Sasaki manifold $S$ is often identified with the submanifold $\{ r=1 \} \subset (C(S), \bar{g})$ and hence the dimension of $S$ is odd. 
Let $\dim S = 2n + 1$. Then, of course, $\dim_{\mathbb{C}}C(S) = n + 1$. 
Let $J$ be the complex structure of the cone $(C(S), \bar{g})$ and define $\tilde{\xi} := J(r\frac{\partial}{\partial r})$. 
The restriction $\xi := \tilde{\xi}|_{\{ r = 1\}}$ of $\tilde{\xi}$ to the submanifold $\{ r = 1\}$ gives a vector field on $S$. 
The vector field $\xi$ is called the {\it Reeb vector field}. 
The $1$-dimensional foliation $\mathcal{F}_{\xi}$ generated by $\xi$ is called the {\it Reeb foliation}. 
Define a differential $1$-form $\eta$ on $S$ by $\eta := g(\xi, \cdot)$. 
Then, one can see that 
\begin{enumerate}
\item $\tilde{\xi}$ is a Killing vector field and satisfies $L_{\tilde{\xi}}J = 0$, 
\item $\nabla_{\xi}\xi = 0$, 
\item $\eta(\xi) = 1$, $\iota_{\xi}d\eta = 0$. 
\end{enumerate}
In particular $\xi$ is a Killing vector field on $S$. 
The $1$-form $\eta$ gives a $2n$-dimensional subbundle $D$ of the tangent bundle $TS$ by 
\begin{equation*}
D = \ker \eta. 
\end{equation*}
The subbundle $D$ is a contact structure of $S$ and there is an 
orthogonal decomposition 
\begin{equation*}
TS = D \oplus L_{\xi}, 
\end{equation*}
where $L_{\xi}$ is the $1$-dimensional trivial bundle generated by the Reeb vector field $\xi$. 

Next we define a section $\Phi$ of the endomorphism bundle ${\rm End}(TS)$ of the tangent bundle $TS$ by $\Phi = \nabla \xi$. 
Then it satisfies that 
\begin{equation*}
\Phi^{2} = -{\rm id} + \eta \otimes \xi
\end{equation*}
and $g(\Phi X, \Phi Y) = g(X, Y) - \eta(X)\eta(Y)$. 
Furthermore, $\Phi |_{D} = J |_{D}$ and $\Phi |_{L_{\xi}} = 0$, and this shows that $\Phi$ gives a complex structure of $D$. 
We call the quadruple $\mathcal{S} = (g, \xi, \eta, \Phi)$ a {\it Sasakian structure} of $S$. 
From these description, the restriction $g_{D} := g |_{D \times D}$ of the metric $g$ to $D$ is an Hermitian metric on $D$ and the associated $2$-form of the Hermitian metric is equal to $\frac{1}{2}d\eta |_{D \times D}$; 
\begin{equation*}
d\eta(X, Y) = 2g(\Phi X, Y)
\end{equation*}
for each $X, Y \in D$. 
Since $\eta$ is a contact form, $\frac{1}{n!}(\frac{1}{2}d\eta)^{n} \wedge \eta$ is a non-vanishing $(2n + 1)$-form and coincides with the Riemannian volume form $dV_{g}$. 
The covariant differentiation of $\Phi$ can be written as a language of the curvature; 
\begin{equation*}
(\nabla_{X}\Phi)(Y) = R(X, \xi)Y = g(\xi, Y)X - g(X, Y)\xi 
\end{equation*}
for any $X, Y \in TS$. 

\subsection{Transverse holomorphic structures and transverse K{\boldmath ${\rm \ddot a}$}hler structures}
As we saw in the last subsection, $\tilde{\xi} - \sqrt{-1}J\tilde{\xi}$ is a holomorphic vector field on $C(S)$. 
Hence there is a $\mathbb{C}^{*}$-action generated by $\tilde{\xi} - \sqrt{-1}J\tilde{\xi}$. 
The local orbits of this action defines a transverse holomorphic structure on the Reeb foliation $\mathcal{F}_{\xi}$ in the following sense; 
There is an open covering $\{ U_{\alpha} \}_{\alpha \in A}$ of $S$ and submersions $\pi_{\alpha} : U_{\alpha} \to V_{\alpha} \subset \mathbb{C}^{n}$ such that when $U_{\alpha} \cap U_{\beta} \neq \phi$ 
\begin{equation*}
\pi_{\alpha} \circ \pi_{\beta}^{-1} : \pi_{\beta}(U_{\alpha} \cap U_{\beta}) \to \pi_{\alpha}(U_{\alpha} \cap U_{\beta})
\end{equation*}
is biholomorphic. 
On each open set $V_{\alpha} \subset \mathbb{C}^{n}$ we can give a K${\rm \ddot a}$hler structure as follows. 
First note that there is a canonical isomorphism $(\pi_{*})_{p}|_{D}: D_{p} \to T_{\pi(p)}V_{\alpha}$ for any $p \in U_{\alpha}$. 
Since $\xi$ generates isometries of $(S, g)$, the restriction $g_{D}$ of the Sasaki metric $g$ to $D$ gives a well-defined Hermitian metric $g^{T}_{\alpha}$ on $V_{\alpha}$. 
This Hermitian structure is in fact K${\rm \ddot a}$hler. 
The fundamental 2-form $\omega^{T}_{\alpha}$ of $g^{T}_{\alpha}$ is the same as the restriction of $\frac{1}{2}d\eta$ to $U_{\alpha}$. 
Hence we see that $\pi_{\alpha} \circ \pi_{\beta}^{-1} : \pi_{\beta}(U_{\alpha} \cap U_{\beta}) \to \pi_{\alpha}(U_{\alpha} \cap U_{\beta})$ gives an isometry of K${\rm \ddot a}$hler manifolds. 
The collection of K${\rm \ddot a}$hler metrics $\{g^{T}_{\alpha}\}_{\alpha \in A}$ on $\{V_{\alpha}\}_{\alpha \in A}$ is called a {\it transverse K$\ddot a$hler metric}. 
Since they are isometric over the overlaps we simply denote by $g^{T}$. 
We also write $\nabla^{T}, R^{T}, Ric^{T}, s^{T}$ for its Levi-Civita connection, the curvature, the Ricci tensor and the scalar curvature. 
By identifying $D_{p}$ and $T_{\pi_{\alpha}(p)}V_{\alpha}$, we have the following formulas for curvature; 
\begin{align}\label{curvature}
R(X, Y, Z, W) &= R^{T}(X, Y, Z, W) + g(\Phi(X), Z)g(\Phi(Y), W) \\ \notag
&\quad - g(\Phi(X), W)g(\Phi(Y), Z) +2g(\Phi(X), Y)g(\nabla_{\xi}Z, W), \\ 
Ric^{T}(X, Y) &= Ric(X, Y) + 2g(X, Y)
\end{align}
for any local sections $X, Y, Z, W$ of $D$. 
For the detail, see \cite{BG}. 
\subsection{Basic forms}
In this section we assume that the Sasaki manifold $(S, g)$ is compact. 
\begin{defn}
A $k$-form $\alpha$ on $S$ is called {\it basic} if 
\begin{equation*}
\iota_{\xi}\alpha = L_{\xi}\alpha = 0. 
\end{equation*}
Let $\Lambda_{B}^{k}$ be the sheaf of germs of basic $k$-forms and $\Omega_{B}^{k}$ be the set of all basic $k$-forms. 
\end{defn}
Let $(x, z^{1}, \cdots, z^{n})$ be a foliation chart on $U_{\alpha}$. 
Consider a complex basic form $\alpha$ which can be written as
\begin{equation*}
\alpha = \alpha_{i_{1}, \cdots, i_{p}, j_{1}, \cdots, j_{q}}dz^{i_{1}} \wedge \cdots \wedge dz^{i_{p}} \wedge d \bar z^{j_{1}} \wedge \cdots \wedge d\bar z^{j_{q}}. 
\end{equation*}
We call such $\alpha$ a {\it basic $(p, q)$-form}. 
It is easy that the definition of basic $(p, q)$-forms is independent of choice of foliation chart. 
Let $\Lambda_{B}^{p, q}$ be the sheaf of germs of basic $(p, q)$-forms and $\Omega_{B}^{p, q}$ be the set of all basic $(p, q)$-forms. 
Then for each $k$, $\Lambda_{B}^{k} \otimes \mathbb{C}$ (resp. $\Omega_{B}^{k}\otimes \mathbb{C}$) can be decomposed as 
\begin{equation*}
\Lambda_{B}^{k} \otimes \mathbb{C} = \oplus_{p + q = k}\Lambda_{B}^{p, q},\ ({\rm resp.\ }\Omega_{B}^{k} \otimes \mathbb{C} = \oplus_{p + q = k}\Omega_{B}^{p, q}). 
\end{equation*}
Since the exterior derivative $d$ preserves the basic forms, its restriction $d_{B}$ to the space of basic forms can be decomposed into $d_{B} = \partial_{B} + \bar \partial_{B}$ by well-defined operators 
\begin{equation*}
\partial_{B} : \Lambda_{B}^{p, q} \to \Lambda_{B}^{p+1, q}\ \text{and}\ \bar{\partial}_{B} : \Lambda_{B}^{p, q} \to \Lambda_{B}^{p, q+1}. 
\end{equation*}
Let $d_{B}^{*}, \partial_{B}^{*}$ and $\bar{\partial}_{B}^{*}$ be the formal adjoint operators of $d_{B}, \partial_{B}$ and $\bar{\partial}_{B}$ and define 
\begin{equation*}
\Delta_{B} := d_{B}^{*}d_{B} + d_{B}d_{B}^{*},\ \Box_{B} := \partial_{B}^{*}\partial_{B} + \partial_{B}\partial_{B}^{*},\ \bar{\Box}_{B} := \bar{\partial}_{B}^{*}\bar{\partial}_{B} + \bar{\partial}_{B}\bar{\partial}_{B}^{*}. 
\end{equation*}
As in the cases of compact K${\rm \ddot a}$hler manifolds, both $\Box_{B}$ and $\bar \Box_{B}$ are real operators and satisfy $\Delta_{B} = \frac{1}{2}\Box_{B} = \frac{1}{2}\bar{\Box}_{B}$ (See \cite{El}). 
Moreover, as shown later $\Delta_{B}$ coincides with Riemannian Laplacian $\Delta$ on the space of basic functions. 
Now we can consider the basic de Rham complex $(\Omega_{B}^{*}, d_{B})$ and the basic Dolbeault complex $(\Omega^{p, *}, \bar{\partial}_{B})$. 
Their cohomology group is called the basic cohomology group. 
Similarly, we can consider the basic harmonic forms. 
El-Kacimi-Alaoui shows in \cite{El} that there is an isomorphism between basic cohomology groups and the space of basic harmonic forms. 

We denote by $C_{B}^{\infty}(S)$ the set of smooth all basic functions on $S$. 
For arbitrary basic function $\varphi \in C_{B}^{\infty}(S)$, define 
\begin{equation*}
\eta_{\varphi} := \eta + 2d_{B}^{c}\varphi, 
\end{equation*}
where $d_{B}^{c} = \frac{\sqrt{-1}}{2}(\bar{\partial}_{B} - \partial_{B})$. 
Then we have 
\begin{equation*}
\frac{1}{2}d\eta_{\varphi} = \frac{1}{2}d\eta + d_{B}d_{B}^{c}\varphi = \frac{1}{2}d\eta + \sqrt{-1}\partial_{B}\bar{\partial}_{B}\varphi. 
\end{equation*}
Thus, for small $\varphi$, $\eta_{\varphi} \wedge (\frac{1}{2}d\eta_{\varphi})^{n}$ is nowhere vanishing and the $1$-form $\eta_{\varphi}$ gives a new Sasakian structure $\mathcal{S}_{\varphi} = (g_{\varphi}, \xi, \eta_{\varphi}, \Phi)$. 
By construction, $\mathcal{S}_{\varphi}$ defines the same transverse holomorphic structure with that of $\mathcal{S}$ (see \cite{FOW} for the detail). 
Under such a deformation, the transverse K${\rm \ddot a}$hler form is deformed in the same basic $(1, 1)$ class $[\frac{1}{2}d\eta]_{B}$. 
We call this class the {\it basic K$\ddot a$hler class}. 
Note that the contact bundle $D$ may be changed under the deformation. 

As we saw in the last subsection, the transverse K${\rm \ddot a}$hler form $\{ \omega^{T}_{\alpha} \}_{\alpha \in A}$ of a Sasaki manifold $(S, g)$ satisfies 
\begin{equation*}
\pi_{\alpha}^{*}\omega^{T}_{\alpha} = \frac{1}{2}d\eta |_{U_{\alpha}}. 
\end{equation*}
Thus they are glued together and give a $d_{B}$-closed basic $(1, 1)$-form $d\eta$ on $S$. 
We also call $\omega^{T} = \frac{1}{2}d\eta$ the transverse K${\rm \ddot a}$hler form. 
Similarly we see that the Ricci forms of the transverse K${\rm \ddot a}$hler metric $\{ \rho^{T}_{\alpha}\}_{\alpha \in A}$, 
\begin{equation*}
\rho^{T}_{\alpha} = -\sqrt{-1}\partial \bar \partial \log \det(g^{T}_{\alpha}), 
\end{equation*}
are glued together and give a $d_{B}$-closed basic $(1, 1)$-form $\rho^{T}$ on $S$. 
$\rho^{T}$ is called the transverse Ricci form. 
Of course, the transverse Ricci form $\rho^{T}$ depends on Sasaki metrics $g$. 
Nevertheless its basic de Rham cohomology class is invariant under deformations of the Sasakian structure by basic functions. 
The basic de Rham cohomology class $[\rho^{T} / 2 \pi]_{B}$ is called the {\it basic first Chern class} and denoted by $c^{B}_{1}(S)$. 
\subsection{Basic first Chern class and Monge-Amp{\boldmath ${\rm \grave{e}}$}re equations}
Let $(S, g)$ be a $(2n + 1)$-dimensional compact Sasaki manifold. 
\begin{defn}
A Sasaki-Einstein manifold is a Sasaki manifold $(S, g)$ with $Ric = 2ng$. 
\end{defn}
The Einstein condition of a Sasaki manifold is translated into Einstein conditions of the Riemannian cone $(C(S), \bar g)$ or the transverse K${\rm \ddot a}$hler structure. 
In short, these conditions are equivalent; 
\begin{enumerate}
\item $g$ is a Sasaki-Einstein metric. 
\item The Riemannian cone $(C(S), \bar g)$ is a Ricci-flat K${\rm \ddot a}$hler manifold. 
\item The transverse K${\rm \ddot a}$hler metric $g^{T}$ satisfies $Ric^{T} = (2n + 2)g^{T}$. 
\end{enumerate}
We say that the basic first Chern class $c_{1}^{B}(S)$ of $S$ is positive if $c_{1}^{B}(S)$ is represented by a transverse K${\rm \ddot a}$hler form, and we express this condition by $c_{1}^{B}(S) > 0$. 
If there exists a Sasaki-Einstein metric, then there exists a transverse K${\rm \ddot a}$hler-Einstein metric $g^{T}$ with ${\rm Ric}^{T} = (2n + 2)g^{T}$ and in particular the basic first Chern class must be positive. 
We remark that there is a further necessary condition for the existence of positive or negative transverse K${\rm \ddot a}$hler-Einstein metric. 
\begin{prop}[Futaki-Ono-Wang, \cite{FOW}]
The basic first Chern class is represented by $\tau d\eta$ for some constant $\tau$ if and only if $c_{1}(D) = 0$. 
\end{prop}

Now we consider a condition for existence of Sasaki-Einstein metric and set up the Monge-Amp${\rm \grave e}$re equation. 
Let $(S, g)$ be a compact Sasaki manifold with Sasakian structure $\mathcal{S} = (g, \xi, \eta, \Phi)$. 
Suppose that $c_{1}^{B}(S) > 0$ and $c_{1}^{B}(S) = (2n + 2)[\frac{1}{2}d\eta]$ (in particular $c_{1}(D) = 0$). 
Then by a result of El Kacimi-Alaoui \cite{El}, there is a unique basic function $h \in C_{B}^{\infty}(S)$ such that 
\begin{equation*}
\rho^{T} - (2n + 2)\frac{1}{2}d\eta = \sqrt{-1}\partial_{B} \bar{\partial}_{B}h,\quad \int_{S}(e^{h} - 1) (\frac{1}{2}d\eta)^{n} \wedge \eta = 0. 
\end{equation*}
Suppose that we can get a Sasaki-Einstein metric by a form $g_{\varphi}$ for some basic function $\varphi$. 
Then associated transverse K${\rm \ddot a}$hler form $\omega^{T}_{\varphi} = \frac{1}{2}d\eta + \sqrt{-1}\partial_{B}\bar{\partial}_{B}\varphi$ satisfies 
\begin{equation*}
\rho^{T}_{\varphi} = (2n + 2)\omega^{T}_{\varphi}. 
\end{equation*}
This leads the transverse K${\rm \ddot a}$hler-Einstein (or equivalently Sasaki-Einstein) equation 
\begin{equation*}
\frac{\det(g^{T}_{i \bar{j}} + \frac{\partial^{2}\varphi}{\partial z^{i}\partial \bar z^{j}})}{\det(g^{T}_{i \bar{j}})} = \exp(-(2n + 2)\varphi + h) 
\end{equation*}
with $(g^{T}_{i \bar{j}} + \frac{\partial^{2}\varphi}{\partial z^{i}\partial \bar z^{j}})$ positive definite. 

In \cite{CFO} and \cite{FOW}, the existence and uniqueness of Sasaki-Einstein metrics on compact toric Sasaki manifold is studied. 
In \cite{FOW}, the authors proved that for any compact toric Sasaki manifold $(S, g)$ with $c_{1}^{B}(S) > 0$ and $c_{1}(D) = 0$, we can get a Sasaki-Einstein metric by deforming the Sasaki structure varying the Reeb vector field (cf. Theorem 1.2. in \cite{FOW}). 
Uniqueness of such Einstein metrics up to a connected group action is proved in \cite{CFO}. 
Given a Sasaki manifold $(S, g)$, we say that another Sasaki metric $g^{\prime}$ on $S$ is {\it compatible} with $g$ if $g$ and $g^{\prime}$ have the same Reeb vector field and the transverse holomorphic structure. 
Note that $g$ and $g^{\prime}$ has the same basic K${\rm \ddot a}$hler class. 
Indeed, for corresponding Sasakian structure $\mathcal{S}^{\prime} = \{ g^{\prime}, \xi^{\prime}, \eta^{\prime}, \Phi^{\prime} \}$, it satisfies that $\zeta := \eta - \eta^{\prime}$ is basic because $\xi = \xi^{\prime}$. 
This shows that $d\eta - d\eta^{\prime} = d\zeta$ and in particular $[\frac{1}{2}d\eta]_{B} = [\frac{1}{2}d\eta^{\prime}]_{B}$. 
Hence by transverse $\partial \bar \partial$-Lemma (see \cite{El}), there exists a basic function $\varphi \in C_{B}^{\infty}(S)$ such that $\frac{1}{2}d\eta^{\prime} = \frac{1}{2}d\eta + \sqrt{-1}\partial_{B} \bar \partial_{B} \varphi$. 
\begin{defn}
The automorphism group of the transverse holomorphic structure of $(S, g)$ is the biholomorphic automorphisms of $C(S)$ which commute with the holomorphic flow generated by $\tilde{\xi} - \sqrt{-1}J\tilde{\xi}$. 
\end{defn}
We denote by ${\rm Aut}(C(S), \tilde{\xi})$ the group of the automorphisms of transverse holomorphic structure and by $G := {\rm Aut}(C(S), \tilde{\xi})_{0}$ its identity component. 
It is known that the action of ${\rm Aut}(C(S), \tilde{\xi})$ on $C(S)$ descends to an action on $S$ preserving the Reeb vector field and the transverse holomorphic structure of the Reeb foliation. 
In particular, $G$ acts on the space of all Sasaki metrics on $S$ which is compatible with $g$. 
The Lie algebra of ${\rm Aut}(C(S), \tilde{\xi})$ is explained as follows. 
\begin{defn}[Futaki-Ono-Wang, \cite{FOW}]
A complex vector field $X$ on $S$ is called a Hamiltonian holomorphic vector field if 
\begin{enumerate}
\item $(\pi_{\alpha})_{*} X$ is a holomorphic vector field on $V_{\alpha}$ for each $\alpha \in A$, 
\item the complex valued function $u_{X} := \sqrt{-1}\eta(X)$ satisfies 
\begin{equation*}
\bar \partial_{B}u_{X} = -\frac{\sqrt{-1}}{2}\iota_{X}d\eta. 
\end{equation*}
\end{enumerate}
\end{defn}
By definition, every Hamiltonian holomorphic vector field is supposed to commute with $\xi$. 
We denote by $\mathfrak{h}$ the set of all Hamiltonian holomorphic vector fields. 
One can check easily that $\mathfrak{h}$ is in fact a Lie algebra. 
Then it is proved in \cite{CFO} that the Lie algebra of ${\rm Aut}(C(S), \tilde{\xi})$ is isomorphic to $\mathfrak{h}$. (For detailed descriptions, see also \cite{FOW}). 
Under the notations and conventions, they proved that, for toric cases, $G$ acts transitively on the space of all Sasaki-Einstein metrics compatible with $g$. 

\subsection{Basic Laplacians for Sasaki manifolds}
In the previous subsection, we introduced the notion of basic Laplacian, which is defined on the space of basic forms. 
Here we shall show that the basic Laplacian $\Delta_{B}$ coincides with the restriction $\Delta|_{C_{B}^{\infty}(S)}$ of the Riemannian Laplacian $\Delta$ to $C_{B}^{\infty}(S)$. 
Let $T := T^{\xi} \subset {\rm Isom}(S, g)$ be the compact subgroup of ${\rm Isom}(S, g)$ generated by the Reeb vector field $\xi$ and $dt$ be the normalized Haar measure on $T$. 
For any smooth function $\varphi \in C^{\infty}(S)$ define 
\begin{equation*}
B(\varphi) := \int_{T}t^{*}\varphi dt. 
\end{equation*}
Then $B$ defines a linear operator on $C^{\infty}(S)$. 
It is clear that $B(\varphi) \in C_{B}^{\infty}(S)$ for any $\varphi \in C^{\infty}(S)$ and $B(\varphi) = \varphi$ if and only if $\varphi \in C_{B}^{\infty}(S)$. 
Furthermore one can show that $B$ is symmetric with respect to the $L^{2}$-inner product on $C^{\infty}(S)$ by Fubini theorem and the symmetry of $T$. 
Hence we obtain a orthogonal decomposition 
\begin{equation*}
C^{\infty}(S) = C_{B}^{\infty}(S) \oplus C_{B}^{\infty}(S)^{\perp}, 
\end{equation*}
where $C_{B}^{\infty}(S)^{\perp}$ is the orthogonal complement of $C_{B}^{\infty}(S)$ with respect to $L^{2}$-inner product and $B$ is the orthogonal projection from $C^{\infty}(S)$ onto $C_{B}^{\infty}(S)$. 

We denote $d^{*}$ the formal adjoint operator of $d$. 
For each $\varphi \in C_{B}^{\infty}(S)$ and $\alpha \in \Omega_{B}^{1}(S)$, we have 
\begin{eqnarray*}
( d_{B}\varphi, \alpha ) 
&=& ( d\varphi, \alpha ) 
= ( \varphi, d^{*}\alpha ) \\
&=& ( B(\varphi), d^{*}\alpha ) 
= ( \varphi, Bd^{*}\alpha ), 
\end{eqnarray*}
where $( \cdot, \cdot )$ is the $L^{2}$-inner product on the space of smooth differential forms. 
This shows that $d_{B}^{*} = B \circ d^{*}$ and hence we obtain 
\begin{equation}\label{L1}
\Delta_{B}\varphi 
= d_{B}^{*}d_{B}\varphi 
= Bd^{*}d\varphi 
= B\Delta \varphi. 
\end{equation}
Furthermore, for each $\varphi \in C_{B}^{\infty}(S)$ and $t \in T$, $t^{*}\Delta \varphi = \Delta t^{*}\varphi = \Delta \varphi$ since $t$ acts on $(S, g)$ as an isometry. 
Therefore we obtain 
\begin{eqnarray}\label{L2}
B\Delta \varphi
= \int_{T}t^{*}\Delta \varphi dt 
= \int_{T}\Delta \varphi dt 
= \Delta \varphi. 
\end{eqnarray}
By combining the equalities (\ref{L1}) and (\ref{L2}), we have the following 
\begin{prop}\label{Laplacian}
For each $\varphi \in C_{B}^{\infty}(S)$ we have 
$\Delta_{B}\varphi = \Delta \varphi. $
\end{prop}

Using the foliation chart, we can get an explicit formula for the basic complex Laplacian $\Box_{B} = \frac{1}{2}\Delta_{B}$ by a similar calculation in K${\rm \ddot a}$hler geometry. 
\begin{prop}
For a foliation chart $(x, z^{1}, \cdots, z^{n})$, we have 
\begin{equation*}
\Box_{B}\varphi = -(g^{T})^{i \bar j}\frac{\partial^{2}\varphi}{\partial z^{i}\partial \bar z^{j}}
\end{equation*}
for each $\varphi \in C_{B}^{\infty}(S)$. 
\end{prop}
\section{A diameter bound for complete Sasaki manifolds}
In this section, we assume that the Sasaki manifold $(S, g)$ is complete. 
A piecewise smooth curve $\gamma: [0, l] \to S$ is called {\it horizontal} if the differential $\dot \gamma(t)$ tangents to $D_{\gamma (t)}$ for all $t \in [0, l]$. 
For each $p,q \in S$, put 
\begin{equation*}
L_{D}(\gamma) := \int_{0}^{l}| \dot \gamma (t) | dt,\quad 
\end{equation*}
and define 
\begin{equation*}
d_{D}(p, q) := \inf\{ L(\gamma)\ |\ \gamma \in \Omega(p, q, D) \}, 
\end{equation*}
where $\Omega(p, q, D)$ is the set of all piecewise smooth horizontal curves joining $p$ to $q$. The function $L_{D}: \Omega(p, q, D) \to \mathbb{R}$ is called the {\it length} of horizontal curves and the function $d_{D}$ on $S \times S$ is called the {\it Carnot-Carath$\acute{e}$odory metric} of $S$. 
For the Riemannian distance function $d$ of $(S, g)$, it is clear that $d \leq d_{D}$. 
Since the contact distribution $D$ is bracket generating (i.e., brackets of local sections of $D$ generates all local sections of $TS$), the classical theorem of Chow tells us that the function $d_{D}$ gives a distance of $S$ and the topology induced by the distance coincides with the original topology of $S$ (For the proof, see \cite{M} for example). 
The main result of this section is stated as follows. 
We say that the transverse Ricci curvature is bounded from below if there exist a constant $\tau \in \mathbb{R}$ such that ${\rm Ric}^{T}(X, X) \geq \tau g(X, X)$ for each $X \in D$. 
We express the condition by $Ric^{T} \geq \tau g^{T}$. 
Hasegawa and Seino shows in \cite{HS} that a complete Sasaki manifold with $Ric^{T} \geq \tau g^{T}$ for a positive constant $\tau > 0$ is compact with finite fundamental group. 
Then we shall show the following stronger result. 
\begin{thm}\label{Myers}
Let $(S, g)$ be a $(2n + 1)$ dimensional complete Sasaki manifold with Sasakian structure $\mathcal{S} = \{g, \xi, \eta, \Phi\}$. 
Suppose that $Ric^{T} \geq \tau g^{T}$ for some constant $\tau > 0$. 
Then 
\begin{equation*}
{\rm diam}(S, d_{D}) \leq 2\pi\sqrt{\frac{2n - 1}{\tau}}. 
\end{equation*}
\end{thm}
Then we can obtain Theorem A immediately because $d \leq d_{D}$. 
Our proof of Theorem \ref{Myers} is based on a variational formula of the energy of normal geodesics on the space of horizontal curves. 

\subsection{Normal geodesics}
A notion of {\it normal geodesics} is defined in sub-Riemannian geometry as the projection on $S$ of solutions of the ``Hamiltonian equation", which is defined below. 

A {\it sub-Riemannian manifold} is a triple $(S, E, g_{E})$ of a smooth manifold $S$, a subbundle $E$ of the tangent bundle $TS$ and a metric $g_{E}$ on $E$. 
For a Sasaki manifold $(S, g)$, the pair of the contact structure $D \subset TS$ and the restriction $g_{D}$ of the Sasaki metric $g$ to $D$ defines a sub-Riemannian structure of $S$, that is, $(S, D, g_{D})$ is a sub-Riemannian manifold. 
Hence we can apply the notions of sub-Riemannian geometry to Sasakian geometry. 
The detailed description can be seen in \cite{M} and \cite{RS1} for example. 
Let $T^{*}S$ be the cotangent bundle of $S$ and $H_{D}: T^{*}S \to \mathbb{R}$ the function on $T^{*}S$ defined by 
\begin{equation*}
H_{D}(p, \alpha) 
:= \frac{1}{2}(g_{D})^{-1}(\alpha |_{D}, \alpha |_{D})
= \frac{1}{2}g^{-1}(\alpha, \alpha ) - \frac{1}{2}\alpha(\xi)^{2} 
\end{equation*}
for each $(p, \alpha) \in T^{*}S$. 
We call the function $H_{D}$ the {\it Hamiltonian function}. 
For any foliation chart $(x_{0}, \cdots, x_{2n})$ with $\frac{\partial}{\partial x_{0}} = \xi$ and the canonical coordinates $(x_{0}, \cdots, x_{2n}$, $\alpha_{0}, \cdots, \alpha_{2n})$ on $T^{*}S$, consider the following ordinary differential equation; 
\begin{equation}\label{Hamiltonian}
\begin{cases}
\; \dot x_{i} = \frac{\partial H_{D}}{\partial \alpha_{i}}, \\
\; \dot \alpha_{i} = -\frac{\partial H_{D}}{\partial x_{i}}. 
\end{cases}  
\end{equation}
We call it the {\it Hamiltonian equation}. 

\begin{defn}
A smooth curve $\gamma : [0, l] \to S$ is called a normal geodesic if there exists a cotangent lift $\Gamma(t) = (\gamma(t), \alpha(t)) : [0, l] \to T^{*}S$ which satisfies the Hamiltonian equation (\ref{Hamiltonian}). 
\end{defn}

By existence and uniqueness of solutions of ordinary differential equations, the Hamiltonian equation (\ref{Hamiltonian}) has unique solution determined by initial value $\Gamma(0) = (p, \alpha) \in T_{p}^{*}S$. 
For a normal geodesic $\gamma(t)$ with the cotangent lift $\Gamma(t) = (\gamma(t), \alpha(t))$, the Hamiltonian equation can be rewritten as 
\begin{equation}\label{Hamiltonian2}
\begin{cases}
\; \dot \gamma(t) = g^{-1}(\alpha) - \alpha(\xi)\xi, \\
\; \frac{d \alpha_{i}}{dt} = -\frac{1}{2}\frac{\partial g^{kj}}{\partial x_{i}}\alpha_{k}\alpha_{j}, 
\end{cases}
\end{equation}
where $g_{kj} := g(\frac{\partial}{\partial x_{k}}, \frac{\partial}{\partial x_{j}})$ is the component of the Sasaki metric $g$ with respect to the local coordinate $(x_{0}, \cdots, x_{2n})$ and $(g^{kj})$ is the inverse matrix of $(g_{kj})$. 
This shows that a normal geodesic is always horizontal. 
Furthermore, the equation (\ref{Hamiltonian2}) implies 
\begin{equation}\label{Hamiltonian3}
\nabla_{\dot \gamma(t)}\dot \gamma(t) = -2\alpha_{0}\Phi(\dot \gamma(t)), 
\end{equation}
where $\alpha_{0} = \alpha(\xi)$ is constant by (\ref{Hamiltonian2}) and by that $\xi$ is a Killing vector field. 
In particular we see that $\gamma(t)$ is constant speed. 

Note that, for a smooth curve $\gamma : [0, l] \to S$ which satisfies the equation (\ref{Hamiltonian3}) for some constant $\alpha_{0} \in \mathbb{R}$, 
we have 
\begin{align*}
\frac{d}{dt}\left( g(\dot \gamma(t), \xi) \right) 
&= g(\nabla_{\dot \gamma(t)}\dot \gamma(t), \xi) + g(\dot \gamma(t), \nabla_{\dot \gamma(t)}\xi) \\
&= -2 \alpha_{0} g( \Phi(\dot \gamma(t)), \xi) + g( \dot \gamma(t), \Phi ( \dot \gamma(t))) \\
&= -2 \alpha_{0} g( \Phi(\dot \gamma(t)), \xi) + \frac{1}{2}d\eta( \dot \gamma(t), \dot \gamma(t)) = 0. 
\end{align*}
Hence we see that $\gamma$ is horizontal if and only if $\dot \gamma (t) \in D_{\gamma(0)}$. 
Now for each smooth horizontal curve $\gamma : [0, l] \to S$ which satisfies the equation (\ref{Hamiltonian3}), define $\alpha(t) := g(\dot \gamma(t) + \alpha_{0}\xi)$ and $\Gamma(t) := (\gamma(t), \alpha(t))$. 
Then we can easily check that the curve $\Gamma(t)$ satisfies the equation (\ref{Hamiltonian}), that is, $\gamma$ is a normal geodesic. 
This shows the following 
\begin{prop}
A smooth curve $\gamma :[0, l] \to S$ is a normal geodesic if and only if it satisfies the equation (\ref{Hamiltonian3}) for some constant $\alpha_{0} \in \mathbb{R}$ and $\dot \gamma(0) \in D_{\gamma(0)}$. 
\end{prop}
A subbundle $E \subset TS$ of the tangent bundle of $S$ is called {\it strong bracket generating} if 
for each $p \in S$ and each nonzero local section $X$ of $E$ around $p$ we have $E_{p} + [X,E]_{p} = T_{p}S$. 
For a Sasaki manifold $(S, g)$, the corresponding contact structure $D$ is strong bracket generating. 
Indeed, for each $p \in S$ and nonzero local section $X$ of $D$ around $p$ we have 
\begin{align*}
g([X, \Phi(X)], \xi) 
&= g(\nabla_{X}\Phi (X), \xi) - g(\nabla_{\Phi (X)}X, \xi) \\
&= -g(\Phi (X), \Phi (X)) + g(X, \Phi^{2}(X)) \\
&= -2g(X, X) \neq 0. 
\end{align*}
This shows $\xi_{p} \in D_{p} + [X, D]_{p}$ and hence we obtain $T_{p}S = D_{p} + [X, D]_{p}$. 

As in the case of Riemannian geometry, every normal geodesic is locally a unique length minimizing curve. 
By the fact that $D$ is strong bracket generating, Strichartz proved that Hopf-Rinow type theorem for the sub-Riemannian manifold $(S, D, g_{D})$ still holds, i.e., any two points on a complete Sasaki manifold can be joined by a length minimizing normal geodesic (See \cite{RS1} and \cite{RS2}). 

\begin{rem}
The assumption that $D$ is strong bracket generating is essential. 
Indeed, for a sub-Riemannian manifold $(S, D, g_{D})$ such that $D$ is {\it not} strong bracket generating, the Hopf-Rinow type theorem does not hold in general. 
There is some examples of length minimizing horizontal curves which are not normal geodesics. 
These examples can be seen in \cite{M}. 
\end{rem}

\subsection{Second Variational formula}
For each $p, q \in S$, consider a functional $E_{D}: \Omega(p, q, D) \to \mathbb{R}$ defined by 
\begin{equation*}
E_{D}(\gamma) := \frac{1}{2}\int_{0}^{l}g(\dot \gamma(t), \dot \gamma(t)) dt, 
\end{equation*}
which is called the {\it energy} of a horizontal curve $\gamma$. 
It is well known in Riemannian geometry, for a constant speed horizontal curve $\gamma$, $\gamma$ minimizes the length functional $L_{D}: \Omega(p, q, D) \to \mathbb{R}$ if and only if it minimizes the energy functional. 
In particular, a length minimizing normal geodesic joining $p$ to $q$ is a energy minimizing curve. 
We shall give a second variational formula of the energy functional on $\Omega(p, q, D)$ for a normal geodesic. 
In this subsection, we assume that every curve $\gamma$ is regular, that is, $\gamma$ is smooth and $|\dot \gamma(t)| \neq 0$ for all $t \in [0, l]$. 

Recall that a {\it variation} of a smooth curve $\gamma : [0, l] \to S$ is a smooth mapping $f: (-\varepsilon, \varepsilon) \times [0, l] \to S$ which satisfies $f(s, 0) = \gamma(0)$, $f(s, l) = \gamma(l)$ and $f(0, t) = \gamma(t)$. 
A smooth vector field $V(t)$ along $\gamma(t)$ is called a {\it variation vector field} of $\gamma$ if it satisfies $V(0) = V(l) = 0$. 
Given a variation $f(s, t)$ of $\gamma$, we can construct a variation vector field $V(t)$ by $V(t) := \frac{\partial f}{\partial s}(s, t)|_{s = 0}$. 
Conversely, for each variation vector field $V(t)$ of $\gamma$, there exists a variation $f(s, t)$ of $\gamma$ whose associated variation vector field is $V(t)$. 

For a horizontal curve $\gamma \in \Omega(p, q, D)$, let $f(s, t): (-\varepsilon, \varepsilon) \times [0, l] \to S$ be a variation of $\gamma$. 
A variation $f(s, t)$ is said to be {\it admissible} if $\frac{\partial f}{\partial t} \in D$ for each $(s, t) \in (-\varepsilon, \varepsilon) \times [0, l]$. 
Similarly, a variation vector field $V(t)$ of $\gamma$ is said to be {\it admissible} if there exists an admissible variation $f(s, t)$ whose variation vector field is $V(t)$. 
A similar argument of Ritor${\rm \acute{e}}$ and Rosales in \cite{RR} tells us that the set $T_{\gamma}\Omega(p, q, D)$ of all admissible variation vector fields of $\gamma$ is given by 
\begin{equation}\label{tangent}
T_{\gamma}\Omega(p, q, D) = \left\{V(t) \in T_{\gamma}\Omega(p, q)\ |\ \frac{d}{dt}g(V(t), \xi) = 2g(V(t), \Phi(\dot \gamma(t))) \right\}, 
\end{equation}
where $T_{\gamma}\Omega(p, q)$ is the set of all variation vector fields of $\gamma$. 

\begin{prop}
Let $\gamma: [0, l] \to S$ be a normal geodesic. 
For each admissible variation $f(s, t)$ of $\gamma$, define $E_{D}(s) := E_{D}(f(s, t))$. 
Then 
\begin{align}\label{secondvariation}
E_{D}^{\prime \prime}(0) 
&= -\int_{0}^{l}g\left( V, \nabla_{\dot \gamma(t)}\nabla_{\dot \gamma(t)}V + R(V, \dot \gamma(t))\dot \gamma(t) \right)dt \\
&\quad + 2\alpha_{0}\int_{0}^{l}\left\{ \eta(V)g(V, \dot \gamma(t)) + g(\nabla_{\dot \gamma(t)}V, \Phi(V))\right\}dt. \notag
\end{align}
\end{prop}
\begin{proof}
At first, we have 
\begin{align*}
E_{D}^{\prime \prime}(s) 
&= \frac{1}{2}\frac{d^{2}}{ds^{2}}\int_{0}^{l}g\left(\frac{\partial f}{\partial t}, \frac{\partial f}{\partial t} \right)dt \\
&= \frac{d}{ds}\int_{0}^{l}g\left(\frac{D}{ds}\frac{\partial f}{\partial t}, \frac{\partial f}{\partial t} \right)dt \\
&= \frac{d}{ds}\int_{0}^{l}g\left(\frac{D}{dt}\frac{\partial f}{\partial s}, \frac{\partial f}{\partial t} \right)dt \\
&= \frac{d}{ds}\int_{0}^{l}\left\{ \frac{d}{dt}g\left(\frac{\partial f}{\partial s}, \frac{\partial f}{\partial t} \right) - g\left(\frac{\partial f}{\partial s}, \frac{D}{dt}\frac{\partial f}{\partial t} \right)\right\}dt \\
&= \int_{0}^{l}\frac{d^{2}}{ds dt}g\left(\frac{\partial f}{\partial s}, \frac{\partial f}{\partial t} \right)dt \\
&\quad  - \int_{0}^{l}g\left(\frac{D}{ds}\frac{\partial f}{\partial s}, \frac{D}{dt}\frac{\partial f}{\partial t} \right)dt - \int_{0}^{l}g\left(\frac{\partial f}{\partial s}, \frac{D}{ds}\frac{D}{dt}\frac{\partial f}{\partial t} \right)dt. 
\end{align*}
By summing the first and third terms, we obtain 
\begin{align*}
&\int_{0}^{l}\frac{d^{2}}{ds dt}g\left(\frac{\partial f}{\partial s}, \frac{\partial f}{\partial t} \right)dt - \int_{0}^{l}g\left(\frac{\partial f}{\partial s}, \frac{D}{ds}\frac{D}{dt}\frac{\partial f}{\partial t} \right)dt \\
&\quad = -\int_{0}^{l}g\left(\frac{\partial f}{\partial s}, \frac{D}{dt}\frac{D}{dt}\frac{\partial f}{\partial s} + R(\frac{\partial f}{\partial s}, \frac{\partial f}{\partial t})\frac{\partial f}{\partial t} \right)dt. \notag
\end{align*}
and hence 
\begin{align}\label{variation1}
E_{D}^{\prime \prime}(s) 
&= -\int_{0}^{l}g\left(\frac{\partial f}{\partial s}, \frac{D}{dt}\frac{D}{dt}\frac{\partial f}{\partial s} + R(\frac{\partial f}{\partial s}, \frac{\partial f}{\partial t})\frac{\partial f}{\partial t} \right)dt \\
&\quad - \int_{0}^{l}g\left(\frac{D}{ds}\frac{\partial f}{\partial s}, \frac{D}{dt}\frac{\partial f}{\partial t} \right)dt. \notag
\end{align}
We shall now calculate the second term of (\ref{variation1}). 
Because $\gamma$ is a normal geodesic, for the integrand we have 
\begin{equation}\label{variation2}
g( \frac{D}{ds}\frac{\partial f}{\partial s}, \frac{D}{dt}\frac{\partial f}{\partial t}) |_{s=0}
= g(\nabla_{V}V, \nabla_{\dot \gamma(t)}\dot \gamma(t)) 
= -2\alpha_{0}g(\nabla_{V}V, \Phi(\dot \gamma(t)))
\end{equation}
by substituting $0$ to $s$. 
To integrate both sides, notice that 
\begin{align*}
\frac{d^{2}}{ds dt}\left( \eta(\frac{\partial f}{\partial s})\right)
&= \frac{d}{ds}\left\{ g\left( \frac{D}{dt}\xi, \frac{\partial f}{\partial s}\right) + \eta(\frac{D}{dt}\frac{\partial f}{\partial s})\right\} \\
&= \frac{d}{ds}\left\{ g\left( \frac{D}{dt}\xi, \frac{\partial f}{\partial s}\right) + \eta(\frac{D}{ds}\frac{\partial f}{\partial t})\right\} \\
&= 2\frac{d}{ds}\left( g\left( \frac{D}{dt}\xi, \frac{\partial f}{\partial s}\right) \right) \\
&= 2\left\{ g\left( \frac{D}{ds}\frac{D}{dt}\xi, \frac{\partial f}{\partial s}\right) + g \left( \frac{D}{dt}\xi, \frac{D}{ds}\frac{\partial f}{\partial s} \right) \right\}. 
\end{align*}
Furthermore, by 
\begin{align*}
g\left( \frac{D}{ds}\frac{D}{dt}\xi, \frac{\partial f}{\partial s}\right) 
&= g\left( \frac{D}{ds}\Phi(\frac{\partial f}{\partial t}), \frac{\partial f}{\partial s}\right) \\
&= g\left( \left( \frac{D}{ds}\Phi \right)(\frac{\partial f}{\partial t}), \frac{\partial f}{\partial s}\right) 
+ g\left( \Phi \left( \frac{D}{ds}\frac{\partial f}{\partial t}\right), \frac{\partial f}{\partial s}\right) \\
&= \eta(\frac{\partial f}{\partial t})g\left( \frac{\partial f}{\partial s}, \frac{\partial f}{\partial s}\right) 
- \eta(\frac{\partial f}{\partial s})g\left( \frac{\partial f}{\partial s}, \frac{\partial f}{\partial t}\right) \\
&\quad + g\left( \Phi \left( \frac{D}{ds}\frac{\partial f}{\partial t}\right), \frac{\partial f}{\partial s}\right) \\
&= - \eta(\frac{\partial f}{\partial s})g\left( \frac{\partial f}{\partial s}, \frac{\partial f}{\partial t}\right) 
+ g\left( \Phi \left( \frac{D}{dt}\frac{\partial f}{\partial s}\right), \frac{\partial f}{\partial s}\right), 
\end{align*}
we obtain the following equality; 
\begin{align}\label{variation3}
\frac{1}{2}\frac{d^{2}}{ds dt}\left( \eta(\frac{\partial f}{\partial s})\right) 
&= -\eta(\frac{\partial f}{\partial s})g\left( \frac{\partial f}{\partial s}, \frac{\partial f}{\partial t}\right) 
+ g\left( \Phi \left( \frac{D}{dt}\frac{\partial f}{\partial s}\right), \frac{\partial f}{\partial s}\right) \\
&\quad + g \left( \frac{D}{dt}\xi, \frac{D}{ds}\frac{\partial f}{\partial s} \right). \notag
\end{align}
Since $\frac{\partial f}{\partial s}(s, 0) = \frac{\partial f}{\partial s}(s,l) = 0$, the integration of both sides of the equality (\ref{variation3}) with respect to $t$ leads us the following equality; 
\begin{equation*}
\int_{0}^{l}g \left( \frac{D}{dt}\xi, \frac{D}{ds}\frac{\partial f}{\partial s} \right)dt 
= \int_{0}^{l}\eta(\frac{\partial f}{\partial s})g\left( \frac{\partial f}{\partial s}, \frac{\partial f}{\partial t}\right) 
- g\left( \Phi \left( \frac{D}{dt}\frac{\partial f}{\partial s}\right), \frac{\partial f}{\partial s}\right). 
\end{equation*}
In particular, by substituting $0$ to $s$, we have 
\begin{align}\label{variation4}
\int_{0}^{l}g \left( \Phi(\dot \gamma(t)), \nabla_{V}V) \right)dt 
&= \int_{0}^{l}\eta(V)g\left( V, \dot \gamma(t) \right) 
- g\left( \Phi \left( \nabla_{\dot \gamma(t)}V\right), V \right) dt \\
&= \int_{0}^{l}\eta(V)g\left( V, \dot \gamma(t) \right) 
+ g\left( \nabla_{\dot \gamma(t)}V, \Phi(V) \right) dt. \notag
\end{align}
Combine the equality (\ref{variation1}), (\ref{variation2}) and (\ref{variation4}), we obtain 
\begin{eqnarray*}
E_{D}^{\prime \prime}(0)
&=& -\int_{0}^{l}g\left(V, \nabla_{\dot \gamma(t)}\nabla_{\dot \gamma(t)}V + R(V, \dot \gamma(t))\dot \gamma(t) \right)dt \\
& &+ 2\alpha_{0}\int_{0}^{l}\left\{ \eta(V)g\left( V, \dot \gamma(t) \right) 
+ g\left( \nabla_{\dot \gamma(t)}V, \Phi(V) \right) \right\}dt, 
\end{eqnarray*}
which is the desired formula. 
\qed \end{proof}

\subsection{A proof of Theorem \ref{Myers}}
Our proof of Theorem \ref{Myers} is based on the classical proof of Myers' theorem. 

Let $p, q$ be an arbitrary pair of points of $S$ and $\gamma : [0, l] \to S$ be a minimizing normal geodesic joining $p$ to $q$ (Since $(S, g)$ is complete, such a normal geodesic always exists). 
We may assume that $|\dot \gamma(t)| = 1$ for all $t \in [0, l]$. 
Then $\gamma$ is a minimizer of $E_{D}: \Omega(p, q, D) \to \mathbb{R}$ and hence $E_{D}^{\prime \prime}(0) \geq 0$ for each admissible variation $f(s, t): (-\varepsilon, \varepsilon) \times [0, l] \to S$ of $\gamma$. 
Choose tangent vectors $X_{1}, \cdots, X_{2(n-1)} \in T_{p}S$ such that $\{X_{1}, \cdots, X_{2(n-1)}, \dot \gamma(0), \Phi(\dot \gamma(0))\}$ is a orthonormal basis of $D$. 
For each $i = 1, \cdots, 2(n-1)$, consider the following linear differential equation for $X_{i}(t) \in D_{\gamma(t)}$ defined by 
\begin{equation*}
\nabla^{T}_{\dot \gamma(t)}X_{i}(t) = \nabla_{\dot \gamma(t)}X_{i}(t) - g(\nabla_{\dot \gamma(t)}X_{i}(t), \xi)\xi = 0\ \text{and}\ X_{i}(0) = X_{i}. 
\end{equation*}
By existence and uniqueness theorem for linear ordinary differential equations, there is a unique global solution $X_{i}(t) \in D$, $t \in [0, l]$. 
\begin{lem}
$\{X_{1}(t), \cdots, X_{2(n-1)}, \dot\gamma(t), \Phi(\dot\gamma(t)) \}$ is a orthonormal basis of $T_{\gamma(t)}S$ for all $t \in [0, l]$. 
\end{lem}
\begin{proof}
First note that $g(\dot \gamma(t), \dot \gamma(t)) = g(\Phi(\dot \gamma(t)), \Phi(\dot \gamma(t))) = 1$ and $g(\Phi(\dot \gamma(t)), \dot \gamma(t)) = 0$. 
Furthermore, since $g(X_{i}, X_{j}) = \delta_{ij}$ and 
\begin{eqnarray*}
\frac{d}{dt}g(X_{i}(t), X_{j}(t)) 
&=& g(\frac{D}{dt}X_{i}(t), X_{j}(t)) + g(X_{i}(t), \frac{D}{dt}X_{j}(t)) \\
&=& g(\nabla^{T}_{\dot \gamma(t)}X_{i}(t), X_{j}(t)) + g(X_{i}(t), \nabla^{T}_{\dot \gamma(t)}X_{j}(t)) = 0, 
\end{eqnarray*}
we see that $g(X_{i}(t), X_{j}(t)) = \delta_{ij}$ for each $t \in [0, l]$. 
Hence it is sufficient to show that $X_{i}(t)$ is perpendicular to both $\dot \gamma(t)$ and $\Phi(\gamma(t))$. 

Define 
\begin{equation*}
f(t) = {^{t}(f_{1}(t), f_{2}(t))} = {^{t}(g(X_{i}(t), \dot \gamma(t)), g(X_{i}(t), \Phi(\dot \gamma(t))))}. 
\end{equation*}
Then, we have 
\begin{align*}
\frac{df_{1}}{dt} 
&= \frac{d}{dt}g(X_{i}(t), \dot \gamma(t)) 
= g(\nabla_{\dot \gamma(t)}X_{i}(t), \dot \gamma(t)) + g(X_{i}(t), \nabla_{\dot \gamma(t)}\dot \gamma(t)) \\
&= -2\alpha_{0}g(X_{i}(t), \Phi(\dot \gamma(t))) = -2\alpha_{0}f_{2}(t), 
\end{align*}
and 
\begin{align*}
\frac{df_{2}}{dt} 
&= \frac{d}{dt}g(X_{i}(t), \Phi(\dot \gamma(t))) 
= g(\nabla_{\dot \gamma(t)}X_{i}(t), \Phi(\dot \gamma(t))) + g(X_{i}(t), \nabla_{\dot \gamma(t)}\Phi(\dot \gamma(t))) \\
&= 2\alpha_{0}g(X_{i}(t), \dot \gamma(t)) = 2\alpha_{0}f_{1}(t). 
\end{align*}
Thus the function $f(t) = {^{t}(f_{1}(t), f_{2}(t))}$ satisfies the following ordinary differential equation; 
\begin{equation*}
\begin{cases}
\; \frac{d}{dt}f(t) 
= \left(
\begin{array}{cc}
0           & -2\alpha_{0}\\
2\alpha_{0} & 0           \\
\end{array}
\right) f(t), \\
\; f(0) = {^{t}(0, 0)}. 
\end{cases}
\end{equation*}
This shows that $f(t) = 0$ for all $t \in [0, l]$ and hence we obtain the desired result. 
\qed \end{proof}
For each $i = 1, \cdots, 2(n-1)$, define $h(t) := \sin\left( \frac{2\pi t}{l} \right)$ and $V_{i}(t) := h(t)X_{i}(t)$. 
Since $X_{i}(t) \in D$ is perpendicular to $\Phi(\dot \gamma(t))$, we see that $V_{i}(t) \in T_{\gamma}\Omega(p, q, D)$. 
Let $f_{i}(s, t)$ be an admissible variation of $\gamma$ whose variation vector field is $V_{i}(t)$. 
Let us calculate the second variation $E_{D}^{\prime \prime}(0) = \frac{d^{2}}{ds}E_{D}(f_{i}(s, t)) \geq 0$ explicitly. 
Note that since $X_{i}(t) \in D$ and $\nabla^{T}_{\dot \gamma(t)}X_{i}(t) = 0$ we have $\eta(V_{i}(t)) = 0$ and 
\begin{align*}
g(\nabla_{\dot \gamma(t)}V_{i}(t), \Phi(V_{i}(t))) 
&= h(t)g(h^{\prime}(t)X_{i}(t) + h(t)\nabla_{\dot \gamma(t)}X_{i}(t), \Phi(X_{i}(t))) \\
&= h(t)h^{\prime}(t)g(X_{i}(t), \Phi(X_{i}(t))) 
+ h^{2}(t)g(\nabla_{\dot \gamma(t)}X_{i}(t), \Phi(X_{i}(t))) \\
&= h(t)h^{\prime}(t)d\eta(X_{i}(t), X_{i}(t)) 
+ h^{2}(t)g(\nabla^{T}_{\dot \gamma(t)}X_{i}(t), \Phi(X_{i}(t))) \\
&= 0. 
\end{align*}
Hence for $V_{i}(t)$ we have 
\begin{equation}\label{variationD1}
E_{D}^{\prime \prime}(0) = -\int_{0}^{l}g\left( V_{i}, \nabla_{\dot \gamma(t)}\nabla_{\dot \gamma(t)}V_{i} + R(V_{i}, \dot \gamma(t))\dot \gamma(t) \right)dt. 
\end{equation}
Now we can calculate easily 
\begin{align*}
\nabla_{\dot \gamma(t)}\nabla_{\dot \gamma(t)}V_{i}(t) 
&= \nabla_{\dot \gamma(t)}\left( h^{\prime}(t)X_{i}(t) + h(t)\nabla_{\dot \gamma(t)}X_{i}(t)\right) \\
&= \nabla_{\dot \gamma(t)}\left( h^{\prime}(t)X_{i}(t) + h(t)g(\nabla_{\dot \gamma(t)}X_{i}(t), \xi)\xi \right) \\
&= h^{\prime \prime}(t)X_{i}(t) + h^{\prime}g(\nabla_{\dot \gamma(t)}X_{i}(t), \xi)\xi \\
&\quad + \frac{d}{dt}\left( h(t)g(\nabla_{\dot \gamma(t)}X_{i}(t), \xi) \right)\xi + h(t)g(\nabla_{\dot \gamma(t)}X_{i}(t), \xi) \Phi(\dot \gamma(t)). 
\end{align*}
Since $V_{i}(t)$ is perpendicular to both $\xi$ and $\Phi(\dot \gamma(t))$, 
we obtain 
\begin{equation}\label{variationD2}
g(V_{i}(t), \nabla_{\dot \gamma(t)}\nabla_{\dot \gamma(t)}V_{i}(t)) = h(t)h^{\prime \prime}(t) = -(\frac{2\pi}{l})^{2}\sin(\frac{2\pi t}{l}). 
\end{equation}
Similarly we have 
\begin{align}\label{variationD3}
g(V_{i}(t), R(V_{i}(t), \dot \gamma(t))\dot \gamma(t)) 
&= \sin^{2}(\frac{2\pi t}{l})g(X_{i}(t), R(X_{i}(t), \dot \gamma(t))\dot \gamma(t)) \\ \notag
&= \sin^{2}(\frac{2\pi t}{l})R(X_{i}(t), \dot \gamma(t), \dot \gamma(t), X_{i}(t)) \\ \notag
&= \sin^{2}(\frac{2\pi t}{l}) R^{T}(X_{i}(t), \dot \gamma(t), \dot \gamma(t), X_{i}(t)) \notag
\end{align}
by equation (\ref{curvature}). 
By substituting (\ref{variationD2}) and (\ref{variationD3}) to (\ref{variationD1}) we obtain the following inequality;
\begin{equation}\label{conclusionD}
0 
\leq E_{D}^{\prime \prime}(0) 
= \int_{0}^{l}\sin^{2}(\frac{2\pi t}{l}) \left\{ (\frac{2\pi}{l})^{2} - R^{T}(X_{i}(t), \dot \gamma(t), \dot \gamma(t), X_{i}(t)) \right\}dt. 
\end{equation}

Next define 
\begin{equation*}
V(t) := h(t)\Phi(\dot \gamma(t)) + k(t)\xi
\end{equation*}
for smooth functions $h(t), k(t) : [0, l] \to \mathbb{R}$ with $h(0) = h(l) = k(0) = k(l) = 0$. 
Then the condition (\ref{tangent}) implies that $V(t) \in T_{\gamma}\Omega(p, q, D)$ if and only if $k^{\prime}(t) = 2h(t)$. 
We suppose that $h(t) := \sin (\frac{2\pi t}{l})$ and $k(t) := \frac{l}{\pi}(1 - \cos (\frac{2 \pi t}{l}))$. 
Then we can easily check that $k^{\prime}(t) = 2h(t)$. 
At first we have 
\begin{align*}
\nabla_{\dot \gamma(t)}V(t) 
&= \nabla_{\dot \gamma(t)}h(t)\Phi(\dot \gamma(t)) + \nabla_{\dot \gamma(t)}k(t)\xi \\
&= h^{\prime}(t)\Phi(\dot \gamma(t)) + h(t)\left\{ (\nabla_{\dot \gamma(t)}\Phi)(\dot \gamma(t)) + \Phi(\nabla_{\dot \gamma(t)}\dot \gamma(t)) \right\} \\
&\quad + k^{\prime}(t)\xi + k(t)\Phi(\dot \gamma(t)) \\
&= h^{\prime}(t)\Phi(\dot \gamma(t)) + h(t)\left( -\xi + 2\alpha_{0} \dot \gamma(t) \right) + k^{\prime}(t)\xi + k(t)\Phi(\dot \gamma(t)) \\
&= (h^{\prime}(t) + k(t))\Phi(\dot \gamma(t)) + h(t)\xi + 2\alpha_{0}h(t)\dot \gamma(t). 
\end{align*}
In addition, by differentiating again we have 
\begin{align*}
\nabla_{\dot \gamma(t)} \left( (h^{\prime}(t) + k(t))\Phi(\dot \gamma(t)) \right) 
&= (h^{\prime \prime}(t) + k^{\prime}(t))\Phi(\dot \gamma(t)) + (h^{\prime}(t) + k(t)) \nabla_{\dot \gamma(t)}\Phi(\dot \gamma(t))\\
&= (h^{\prime \prime}(t) + k^{\prime}(t))\Phi(\dot \gamma(t)) \\
&\quad + (h^{\prime}(t) + k(t)) ( -\xi + 2\alpha_{0}\dot \gamma(t))\\
&= (h^{\prime \prime}(t) + k^{\prime}(t))\Phi(\dot \gamma(t)) \\
&\quad - (h^{\prime}(t) + k(t))\xi + 2\alpha_{0}(h^{\prime}(t) + k(t))\dot \gamma(t), 
\end{align*}
and 
\begin{align*}
\nabla_{\dot \gamma(t)}2\alpha_{0}h(t) \dot \gamma(t) 
&= 2\alpha_{0}\left( h^{\prime}(t) \dot \gamma(t) + h(t)\nabla_{\dot \gamma(t)} \dot \gamma(t)\right) \\
&= 2\alpha_{0}h^{\prime}(t) \dot \gamma(t) - (2\alpha_{0})^{2}h(t)\Phi( \dot \gamma(t)). 
\end{align*}
By combining them we obtain 
\begin{align*}
\nabla_{\dot \gamma(t)}\nabla_{\dot \gamma(t)}V(t) 
&= \left( h^{\prime \prime}(t) + 3h(t) - (2\alpha_{0})^{2}h(t) \right)\Phi(\dot \gamma(t)) \\
&\quad + 2\alpha_{0}(2h^{\prime}(t) + k(t))\dot \gamma(t) - k(t)\xi 
\end{align*}
and hence 
\begin{equation}\label{variationDD1}
g(V(t), \nabla_{\dot \gamma(t)}\nabla_{\dot \gamma(t)}V(t)) 
= h(t) \left( h^{\prime \prime}(t) + 3h(t) - (2\alpha_{0})^{2}h(t) \right) -k^{2}(t). 
\end{equation}
For the curvatures we have 
\begin{align}\label{variationDD2}
g(V(t), R(V(t), \dot \gamma(t))\dot \gamma(t)) 
&= h^{2}(t)g(\Phi(\dot \gamma(t)), R(\Phi(\dot \gamma(t)), \dot \gamma(t))\dot \gamma(t)) \\ \notag
&\quad + 2h(t)k(t)g(\Phi(\dot \gamma(t)), R(\xi, \dot \gamma(t))\dot \gamma(t)) \\ \notag
&\quad + k^{2}(t)g(\xi, R(\xi, \dot \gamma(t))\dot \gamma(t)) \\ \notag
&= h^{2}(t)R(\Phi(\dot \gamma(t)), \dot \gamma(t), \dot \gamma(t), \Phi(\dot \gamma(t))) + k^{2}(t). \\ \notag
&= h^{2}(t)R^{T}(\Phi(\dot \gamma(t)), \dot \gamma(t), \dot \gamma(t), \Phi(\dot \gamma(t))) - 3h^{2}(t) + k^{2}(t). \\ \notag
\end{align}
By substituting (\ref{variationDD1}) and (\ref{variationDD2}) to (\ref{secondvariation}), we obtain 
\begin{align}\label{conclusionDD}
0 &\leq E_{D}^{\prime \prime}(0) \\ \notag
&= -\int_{0}^{l}\left\{ h(t) \left( h^{\prime \prime}(t) + 3h(t) - (2\alpha_{0})^{2}h(t) \right) -k^{2}(t)\right\}dt \\ \notag
&\quad - \int_{0}^{l}\left\{ h^{2}(t)R^{T}(\Phi(\dot \gamma(t)), \dot \gamma(t), \dot \gamma(t), \Phi(\dot \gamma(t))) - 3h^{2}(t) + k^{2}(t) \right\}dt \\ \notag
&\quad - (2\alpha_{0})^{2}\int_{0}^{l}h^{2}(t) dt \\ \notag
&= -\int_{0}^{l} \left\{ h(t)h^{\prime \prime}(t) + h^{2}(t) R^{T}(\Phi(\dot \gamma(t)), \dot \gamma(t), \dot \gamma(t), \Phi(\dot \gamma(t)))\right\}dt \\ \notag
&= \int_{0}^{l} \sin^{2}(\frac{2\pi t}{l})\left\{ (\frac{2\pi}{l})^{2} - R^{T}(\Phi(\dot \gamma(t)), \dot \gamma(t), \dot \gamma(t), \Phi(\dot \gamma(t))) \right\}dt. \notag
\end{align}
Finally, by summing (\ref{conclusionD}) and (\ref{conclusionDD}), we obtain 
\begin{equation*}
0 \leq \int_{0}^{l}\sin^{2}(\frac{2\pi t}{l})\left\{ (\frac{2\pi}{l})^{2}(2n - 1) - Ric^{T}(\dot \gamma(t), \dot \gamma(t)) \right\}dt. 
\end{equation*}
Furthermore, by assumption $Ric^{T} \geq \tau g^{T}$, 
\begin{equation*}
0 \leq \int_{0}^{l}\sin^{2}(\frac{2\pi t}{l})\left\{ (\frac{2\pi}{l})^{2}(2n - 1) - \tau \right\}dt. 
\end{equation*}
This shows that $0 \leq (\frac{2\pi}{l})^{2}(2n - 1) - \tau$ and hence 
\begin{equation*}
d_{D}(p, q) = l \leq 2\pi \sqrt{\frac{2n - 1}{\tau}}. 
\end{equation*}
Hence we obtain ${\rm diam}(S, d_{D}) \leq 2\pi \sqrt{\frac{2n - 1}{\tau}}$ and this completes the proof of Theorem \ref{Myers}. 

\section{A proof of Theorem B}
\subsection{Generalized Aubin's equation}
In this section we give a proof of Theorem B. 
Our proof is based on the arguments of Bando and Mabuchi in \cite{BM}. 
Throughout this section, we denote $(S, g)$ by a $(2n + 1)$-dimensional compact Sasaki manifold with $c_{1}^{B}(S) > 0$ and $c_{1}(D) = 0$ and $\mathcal{S} = (g, \xi, \eta, \Phi)$ by the associated Sasakian structure. 
By assumption, we may assume $[\rho^{T}]_{B} = (2n + 2)[\omega^{T}]_{B}$. 
Put $\mathscr{S}(g)$ to be the set of all Sasaki metric on $S$ which is compatible with $g$ and $\mathscr{H} := \{ \varphi \in C_{B}^{\infty}(S)\ |\ (g_{i \bar j}^{T} + \frac{\partial^{2} \varphi}{\partial z^{i}\partial \bar z^{j}})\ \text{is positive definite} \}$. 
Clearly $g_{\varphi} \in \mathscr{S}(g)$ for each $\varphi \in \mathscr{H}$. 
We denote by $\mathscr{E}$ the set of all Sasaki-Einstein metrics in $\mathscr{S}(g)$. 
Throughout this section, we assume $\mathscr{E} \neq \phi$. 

Let $V := \int_{S}(\frac{1}{2}d\eta)^{n} \wedge \eta$ and define the functionals $L_{\eta}$, $M_{\eta}$, $I_{\eta}$ and $J_{\eta}$ on $\mathscr{H}$ by 
\begin{align*}
L_{\eta}(\varphi) 
&:= \frac{1}{V}\int_{a}^{b}dt\int_{S}\dot{\varphi_{t}}(\frac{1}{2}d\eta_{\varphi_{t}})^{n} \wedge \eta_{\varphi_{t}} , \\
M_{\eta}(\varphi) 
&:= -\frac{1}{V}\int_{a}^{b}dt\int_{S}\dot{\varphi_{t}}(s^{T}(\varphi) - n(2n + 2))(\frac{1}{2}d\eta_{\varphi_{t}})^{n} \wedge \eta_{\varphi_{t}}, \\
I_{\eta}(\varphi) 
&:= \frac{1}{V}\int_{S}\varphi\left( (\frac{1}{2}d\eta)^{n} \wedge \eta - (\frac{1}{2}d\eta_{\varphi})^{n} \wedge \eta_{\varphi}\right), \\
J_{\eta}(\varphi) 
&:= \frac{1}{V}\int_{a}^{b}dt \int_{S}\dot{\varphi_{t}}\left( (\frac{1}{2}d\eta)^{n} \wedge \eta - (\frac{1}{2}d\eta_{\varphi_{t}})^{n} \wedge \eta_{\varphi_{t}}\right), 
\end{align*}
where $\{\varphi_{t}\ |\ t \in [a, b]\}$ is an arbitrary piecewise smooth path in $\mathscr{H}$ such that $\varphi_{a} = 0$ and $\varphi_{b} = \varphi$. 
These are the ``Sasaki version" of the functionals defined on the space of K${\rm \ddot a}$hler potentials in \cite{BM} and have the similar properties to those. 
The precise definitions and basic properties can be seen in the Appendix. 

Since $[\rho^{T}]_{B} = (2n + 2)[\omega^{T}]_{B}$, there exists a unique basic function $h \in C_{B}^{\infty}(S)$ which satisfies $\rho^{T} - (2n + 2)\omega^{T} = \sqrt{-1}\partial_{B}\bar{\partial}_{B}h$ and $\int_{S}(e^{h} -1)(\frac{1}{2}d\eta)^{n} \wedge \eta = 0$. 
Consider the following one-parameter families of equations; 
\begin{align}
\frac{\det(g^{T}_{i \bar{j}} + \frac{\partial^{2}\psi_{t}}{\partial z^{i}\partial \bar z^{j}})}{\det(g^{T}_{i \bar{j}})} 
&= \exp(-t(2n+2)\psi_{t} + h);\quad t \in [0, 1], \label{MA1} \\
\frac{\det(g^{T}_{i \bar{j}} + \frac{\partial^{2}\varphi_{t}}{\partial z^{i}\partial \bar z^{j}})}{\det(g^{T}_{i \bar{j}})} 
&= \exp(-t(2n+2)\varphi_{t} -L_{\eta}(\varphi_{t}) + h);\quad t \in [0, 1], \label{MA2}
\end{align}
where solutions $\psi_{t}$ and $\varphi_{t}$ are both required to belong to $\mathscr{H}$. 
Note that, for both equations, these are just the transverse K${\rm \ddot a}$hler-Einstein equation at $t=1$. 
As a remark in \cite{BM}, there is no difference between (\ref{MA1}) and (\ref{MA2}) in finding solutions for $t \neq 0$. 
\begin{rem}\label{remark_positive}
Choose an arbitrary $t \in [0, 1]$. 
Let $\psi_{t}$ (resp. $\varphi_{t}$) be a solution of (\ref{MA1}) (resp. (\ref{MA2})) and $g_{t}$ be the Sasaki metric corresponding to the Sasaki structure $\eta_{\psi_{t}}$ (resp. $\eta_{\varphi_{t}} $). 
Then $g_{t}$ satisfies $\rho_{t}^{T} = t(2n + 2)\omega^{T}_{t} + (1-t)(2n + 2)\omega^{T}$, and in particular we have $\rho_{t}^{T} - t(2n + 2)\omega^{T}_{t} \geq 0$. 
Furthermore if $t \neq 0$, then $\rho_{t}^{T} - t(2n + 2)\omega^{T}_{t}$ is strictly positive. 
\end{rem}
We first consider the existence of the equation (\ref{MA2}) at $t = 0$. 
For the equation (\ref{MA1}), a result of El-Kacimi-Alaoui \cite{El} guarantees the existence of a solution at $t = 0$. 
Then the existence and uniqueness of a solution of the equation (\ref{MA2}) follows immediately.
\begin{thm}[El Kacimi-Alaoui, \cite{El}]\label{El}
If $t = 0$, then the equation (\ref{MA1}) has a solution which is unique up to an additive constant. 
\end{thm} 
\begin{cor}\label{existence0}
The equation (\ref{MA2}) has a unique solution $\varphi_{0}$ at $t = 0$. 
The solution $\varphi_{0}$ satisfies $L_{\eta}(\varphi_{0}) = 0$. 
\end{cor}
\begin{proof}
Take any solution $\psi_{0} \in \mathscr{H}$ of the equation (\ref{MA1}) at $t = 0$ and define 
$\varphi_{0} := \psi_{0} - L_{\eta}(\psi_{0})$. 
Then it is easy to check that $\varphi_{0}$ is a solution of the equation (\ref{MA2}). 
This proves the existence of a solution of (\ref{MA2}). 
Furthermore, for any solution $\varphi_{0}$ of (\ref{MA2}) we have 
\begin{align*}
\int_{S}(\frac{1}{2}d\eta )^{n} \wedge \eta 
&= \int_{S}(\frac{1}{2}d\eta_{\varphi} )^{n} \wedge \eta_{\varphi} \\
&= \int_{S}\exp(-L_{\eta}(\varphi_{0}) + h)(\frac{1}{2}d\eta )^{m} \wedge \eta \\
&= \exp(-L_{\eta}(\varphi_{0}))\int_{S}e^{h}(\frac{1}{2}d\eta )^{n} \wedge \eta \\
&= \exp(-L_{\eta}(\varphi_{0}))\int_{S}(\frac{1}{2}d\eta )^{n} \wedge \eta. 
\end{align*}
This shows that $L_{\eta}(\varphi_{0}) = 0$. 
Therefore, $\varphi_{0}$ is also a solution of equation (\ref{MA1}) at $t = 0$. 
Now the required uniqueness now follows from Theorem \ref{El} and that $L_{\eta}(\varphi_{0}) = 0$. 
\qed \end{proof}
For each $\varphi \in \mathscr{H}$, we denote by $\Box_{\varphi} := \Box_{B,g_{\varphi}}$ the basic complex Laplacian with respect to the Sasaki metric $g_{\varphi}$. 
The following proposition shows the local extension property of solutions of (\ref{MA2}) for $t \in [0, 1)$ (see also \cite{S}). 
\begin{prop}\label{openness}
Let $0 < \tau < 1$. 
Suppose that the equation (\ref{MA2}) has a solution $\varphi_{\tau}$ at $t = \tau$. Then for some $\varepsilon > 0$, $\varphi_{\tau}$ uniquely extends to a smooth one parameter family $\{ \varphi_{t}\ |\ t \in [0, 1] \cap [\tau - \varepsilon, \tau + \varepsilon] \}$ of solutions of (\ref{MA2}). 
\end{prop}
\begin{proof}
Let $2 \leq k \in \mathbb{Z}$ and fix $\alpha \in \mathbb{R}$ with $0 < \alpha < 1$. 
Let $C_{B}^{k, \alpha}(S)$ be the set of all basic functions which belong to $C^{k, \alpha}(S)$, and $\mathscr{H}^{k, \alpha}$ be the open set of all functions $\varphi \in C_{B}^{k, \alpha}(S)$ satisfying that $(g^{T}_{i \bar j} + \frac{\partial^{2} \varphi}{\partial z^{i} \partial \bar z^{j}})$ is positive definite. 
Define $\Gamma : \mathscr{H}^{k, \alpha} \times \mathbb{R} \to C_{B}^{k-2, \alpha}(S)$ by 
\begin{equation*}
\Gamma(\varphi, t) := \log \left( \frac{\det(g^{T}_{i \bar{j}} + \frac{\partial^{2}\varphi}{\partial z^{i}\partial \bar z^{j}})}{\det(g^{T}_{i \bar{j}})} \right) + t(2n+2)\varphi + L_{\eta}(\varphi) - h. 
\end{equation*}
Then its Fr${\rm \acute e}$chet derivative $D_{\varphi}\Gamma$ with respect to the first factor at $(\varphi, t)$ is given by 
\begin{equation*}
D_{\varphi}\Gamma(\psi) = (-\Box_{\varphi} + t(2n + 2))\psi + \frac{1}{V}\int_{S}\psi (\frac{1}{2}d\eta_{\varphi})^{n} \wedge \eta_{\varphi}
\end{equation*}
for each $\psi \in C_{B}^{k, \alpha}(S)$. 
Note that, by the well-known regularity theorem, we have $\varphi \in C_{B}^{\infty}(S)$ for every $(\varphi, t) \in \mathscr{H}^{k, \alpha} \times \mathbb{R}$ whenever $\Gamma(\varphi, t) = 0$. 
Since $\Gamma(\varphi_{\tau}, \tau) = 0$, an application of the implicit function theorem now reduces the proof to showing that $D_{\varphi}\Gamma$ is invertible at $(\varphi_{\tau}, \tau)$. 
There are the following cases. 
\begin{description}
\item[Case 1] $\tau = 0$. 
Then $D_{\varphi} \Gamma$ at $(\varphi_{0}, 0)$ is given by 
\begin{equation*}
D_{\varphi}\Gamma(\psi) = -\Box_{\varphi_{0}}\psi + \frac{1}{V}\int_{S}\psi (\frac{1}{2}d\eta_{\varphi_{0}})^{n} \wedge \eta_{\varphi_{0}}, 
\end{equation*}
which is invertible. 
\item[Case 2] $\tau \neq 0$. 
First note $\rho_{\varphi_{\tau}}^{T} > \tau(2n + 2)\omega_{\varphi_{\tau}}^{T}$ by Remark \ref{remark_positive}. Then the similar argument of Lichnerowicz \cite{Li} tells us that the first positive eigenvalue of $\Box_{\varphi_{\tau}}$ is greater than $\tau(2n + 2)$ (see also the proof of Theorem 2.4.3 in \cite{F}). 
This shows that $D_{\varphi}\Gamma |_{(\varphi_{\tau}, \tau)}$ is invertible. 
\end{description}
\qed \end{proof}
\begin{rem}\label{openness2}
A Hamiltonian holomorphic vector field $X$ is said to be {\it normalized} if the Hamiltonian function $u_{X}$ satisfies that 
\begin{equation*}
\int_{S}u_{X}e^{h}(\frac{1}{2}d\eta)^{n} \wedge \eta = 0. 
\end{equation*}
For any $X \in \mathfrak{h}$, there exists a constant $c$ such that $X + c\xi$ is normalized Hamiltonian holomorphic vector field. 
We denote by $\mathfrak{h}_{0}$ the set of all normalized Hamiltonian holomorphic vector fields. 
If $\mathfrak{h}_{0} = \{ 0 \}$ and $\tau =1$, the result of Futaki, Ono and Wang (cf. Theorem 5.1 in \cite{FOW}) tells us that $\ker(\Box_{\varphi_{1}} - (2n + 2)) \cong \mathfrak{h}_{0} = \{ 0 \}$ and the first positive eigenvalue of $\Box_{\varphi_{1}}$ is greater than $2n +2$. 
This shows that $D_{\varphi}\Gamma |_{(\varphi_{1}, 1)}$ is invertible. 
Hence we obtain that Proposition \ref{openness} still holds for the case that $\mathfrak{h}_{0} = \{ 0 \}$ and $\tau =1$. 
\end{rem}

Next we shall give a bound for solutions of (\ref{MA2}). 
By El Kacimi-Alaoui's generalization of Yau's estimate \cite{Yau} for transverse Monge-Ampe${\rm \grave r}$e equations, the $C^{0}$-estimate for solutions $\varphi$ of (\ref{MA2}) implies the $C^{2, \alpha}$-estimate for them. 
First of all, we give a bound for the oscillation ${\rm osc}_{S}\varphi = \sup_{S}\varphi - \inf_{S}\varphi$ for $\varphi \in \mathscr{H}$. 
The following proposition is proved by the same way as K${\rm \ddot a}$hler geometry. 
\begin{prop}\label{osc1}
Let $\varphi \in \mathscr{H}$. 
We assume that there exists real constants $A, \delta > 0$ such that 
\begin{equation*}
\| \psi \|_{L^{2m/(m-1)}} \leq A\| d\psi \|_{L^{2}},\ \delta \| \psi \|_{L^{2}}^{2} \leq \| d\psi \|_{L^{2}}^{2} 
\end{equation*}
for every basic function $\psi \in C_{B}^{\infty}(S)$ which satisfies $\int_{S}\psi dV_{g} = 0$. 
Moreover, suppose that 
\begin{equation*}
\sup_{S}\frac{\det (g_{i \bar{j}}^{T} + \frac{\partial^{2}\varphi}{\partial z^{i}\partial \bar z^{j}})}{\det (g_{i \bar{j}}^{T})} \leq B
\end{equation*}
for some constant $B >0$. 
Then there exists a real constant $C > 0$ depending only $A, \delta$ and $B$ which satisfies 
\begin{equation*}
{\rm osc}_{S}\varphi \leq C. 
\end{equation*}
\end{prop}

Proposition \ref{osc1} has the following important implication. 
In our proof, Theorem A is essential to obtain a bound for the infimum of basic functions $\varphi \in \mathscr{H}$. 

\begin{prop}\label{osc}
Let $G = G_{\eta}$ be the Green function of the initial metric $g$ and $K$ be the real constant which satisfies $\inf G \geq -K$. 
For $\varphi \in \mathscr{H}$, assume that $\rho_{\varphi}^{T} \geq t(2n + 2)\omega_{\varphi}^{T}$ for some $t \in (0, 1]$. 
Then there exists a positive constant $\gamma >0$ such that 
\begin{eqnarray*}
{\rm osc}_{S}\varphi 
\leq I_{\eta}(\varphi) + 2n\left( KV_{0} + \gamma \frac{(2\pi)^{2}(2n-1)}{t(2n + 2)}\right), 
\end{eqnarray*}
where $V_{0} := V/n!$ is the volume of $(S, g)$. 
\end{prop}

\begin{proof}
First we observe that, by the identity $\frac{1}{2}d\eta_{\varphi} = \frac{1}{2}d\eta + \sqrt{-1} \partial_{B} \bar \partial_{B}\varphi$, we have 
\begin{equation*}
\Box_{0}\varphi \leq n\ \text{and}\ \Box_{\varphi}\varphi \geq -n. 
\end{equation*}
Since the basic Laplacian coincides with the restriction of the Riemannian Laplacian to $C_{B}^{\infty}(S)$ (cf. Proposition \ref{Laplacian}), 
we have 
\begin{align*}
\varphi(p) 
&= \frac{1}{V_{0}}\int_{S}\varphi dV_{g} 
+ \int_{S}(G(p, q) + K)(\Delta_{0}\varphi)(q) dV_{g}(q) \\
&= \frac{1}{V}\int_{S}\varphi (\frac{1}{2}d\eta)^{n} \wedge \eta 
+ \int_{S}(G(p, q) + K)(2\Box_{0}\varphi)(q) dV_{g}(q) \\
&\leq \frac{1}{V}\int_{S}\varphi (\frac{1}{2}d\eta)^{n} \wedge \eta 
+ 2nKV_{0}. 
\end{align*}
This leads the following estimate for $\varphi$; 
\begin{equation}\label{sup}
\sup_{S}\varphi \leq \frac{1}{V}\int_{S}\varphi (\frac{1}{2}d\eta)^{n} \wedge \eta + 2nKV_{0}. 
\end{equation}

On the other hand, by using the Green function $G_{\varphi}$ of $g_{\varphi}$ we have 
\begin{align*}
\varphi(p) 
&= \frac{1}{V_{0}}\int_{S}\varphi dV_{g_{\varphi}} 
+ \int_{S}(G_{\varphi}(p, q) + K_{\varphi})(\Delta_{\varphi} \varphi)(q) dV_{g_{\varphi}}(q) \\
&= \frac{1}{V}\int_{S}\varphi (\frac{1}{2}d\eta_{\varphi})^{n} \wedge \eta_{\varphi}
+ \int_{S}(G_{\varphi}(p, q) + K_{\varphi})(2\Box_{\varphi} \varphi)(q) dV_{g_{\varphi}}(q), 
\end{align*}
where $K_{\varphi} = \sup(-G_{\varphi})$. 
Since ${\rm Ric}^{T} \geq t(2n + 2)$ by assumption, we have ${\rm Ric} \geq t(2n + 2) - 2 \geq -2$. 
Then Theorem 3.2 in \cite{BM} tells us that there exists a positive constant $\gamma > 0$ which depends only $n$ and satisfies 
\begin{align*}
\varphi(p) 
&\geq \frac{1}{V}\int_{S}\varphi (\frac{1}{2}d\eta_{\varphi})^{n} \wedge \eta_{\varphi} 
- 2n\gamma {\rm diam}(S, g_{\varphi})^{2}. 
\end{align*}
Moreover, by Theorem \ref{Myers2} we have 
\begin{equation*}
{\rm diam}(S, g_{\varphi}) \leq 2\pi\sqrt{\frac{2n-1}{t}}. 
\end{equation*}
Hence we obtain 
\begin{equation}\label{inf}
\inf_{S}\varphi \geq \frac{1}{V}\int_{S}\varphi (\frac{1}{2}d\eta_{\varphi})^{n} \wedge \eta_{\varphi} 
- 2n\gamma \frac{(2\pi)^{2}(2n-1)}{t}
\end{equation}
and hence 
\begin{equation}
{\rm osc}_{S}\varphi 
= \sup_{S}\varphi - \inf_{S}\varphi 
\leq I_{\eta}(\varphi) + 2n\left( KV_{0} + (2n-1) \gamma \frac{(2\pi)^{2}}{t}\right)
\end{equation}
by inequalities (\ref{sup}) and (\ref{inf}). 
\end{proof}

We then see that a bound for $I_{\eta}$ on solutions of (\ref{MA2}) implies a priori $C^{0}$-estimate for solutions. 
\begin{prop}\label{a priori}
Let $\varphi_{t}$ be a solution of (\ref{MA2}) at $t$ and $A > 0$ be a constant which satisfies 
\begin{equation*}
I_{\eta}(\varphi_{t}) \leq A. 
\end{equation*}
Then there exists a real constant $C > 0$ depending only $A$, $n$ and the initial metric $g$ which satisfies 
\begin{equation*}
\sup_{S} \varphi_{t} \leq C. 
\end{equation*}
\end{prop}
\begin{proof}
By Proposition \ref{osc}, there exists $C_{1} > 0$ which depends only $A$, $n$ and the initial metric $g$ such that 
\begin{equation*}
t{\rm osc}_{S}\varphi_{t} 
\leq t \left( I_{\eta}(\varphi_{t}) + 2n\left( K_{0}V_{0} + \gamma \frac{(2\pi)^{2}(2n - 1)}{t(2n + 2)}\right)\right)
\leq C_{1}, 
\end{equation*}
where $K_{0}$ is a constant which satisfies $\inf G_{\eta} \geq -K_{0}$. 
By integrating both sides of (\ref{MA2}) we have 
\begin{equation*}
\int_{S}\exp(-t(2n + 2)\varphi_{t} - L_{\eta}(\varphi_{t}) + h) (\frac{1}{2}d\eta)^{n} \wedge \eta 
= \int_{S}(\frac{1}{2}d\eta)^{n} \wedge \eta 
\end{equation*}
and hence 
\begin{equation}\label{integration}
-t(2n + 2)\varphi_{t}(p_{t}) - L_{\eta}(\varphi_{t}) + h(p_{t}) = 0
\end{equation}
for some $p_{t} \in S$. 
Then there exists a constant $C_{2} > 0$ which depends only $A$, $n$ and the initial metric $g$ such that 
\begin{align*}
|-t(2n + 2)\varphi_{t}(p) - L_{\eta}(\varphi_{t}) + h(p)|
&\leq |t(2n + 2)\varphi_{t}(p_{t}) - t(2n + 2)\varphi_{t}(p)| \\
&\quad + |h(p) - h(p_{t}) | \\
&\leq t(2n + 2){\rm osc}_{S}\varphi_{t} + 2\sup_{S}|h| \leq C_{2}
\end{align*}
for each $p \in S$. 
This shows that 
\begin{equation*}
\sup_{S}\left| \log \frac{\det(g^{T}_{i \bar{j}} + \frac{\partial^{2}\varphi_{t}}{\partial z^{i}\partial \bar z^{j}})}{\det(g^{T}_{i \bar{j}})} 
\right|
= \left| \exp(-t(2n+2)\varphi_{t} -L_{\eta}(\varphi_{t}) + h) \right| \leq C_{2}. 
\end{equation*}
Hence by Proposition \ref{osc1} there exists a constant $C_{3} > 0$ such that 
\begin{equation}\label{osc_inequality}
{\rm osc}_{S}\varphi \leq C_{3}. 
\end{equation}
For $p_{t}$ defined above we have 
\begin{align*}
|L_{\eta}(\varphi_{t}) - \varphi_{t}(p_{t})| 
&= \left| \frac{1}{V}\int_{0}^{1} ds \int_{S}(\varphi_{t} - \varphi_{t}(p_{t}))(\frac{1}{2}d\eta_{s\varphi_{t}})^{n} \wedge \eta \right| \\
&\leq \frac{1}{V} \int_{0}^{1} ds \int_{S} {\rm osc}_{S}\varphi_{t} (\frac{1}{2}d\eta_{s\varphi_{t}})^{n} \wedge \eta  \\
&= {\rm osc}_{S}\varphi_{t} \leq C_{3}. 
\end{align*}
Then by combining (\ref{integration}), we obtain 
\begin{align*}
\{1 + t(2n + 2)\} | \varphi_{t}(p_{t}) |
&\leq | \varphi_{t}(p_{t}) - L_{\eta}(\varphi_{t}) + h(p_{t}) | \\
&\leq | \varphi_{t}(p_{t}) - L_{\eta}(\varphi_{t})| + |h(p_{t}) | \\
&\leq C_{3} + \sup_{S}|h|
\end{align*}
and hence 
\begin{equation*}
\sup_{S}\varphi_{t}
= {\rm osc}_{S}\varphi_{t} + \inf_{S}\varphi_{t} 
\leq {\rm osc}_{S}\varphi + \varphi_{t}(p_{t}) 
\leq 2C_{3} + \sup_{S}|h|
\end{equation*}
If we put $C := 2C_{3} + \sup_{S}|h|$, then it depends only $A$, $n$ and the initial metric $g$, and satisfies $\sup_{S}\varphi_{t} \leq C$. 
This completes the theorem. 
\end{proof}

To obtain a bound for $I_{\eta}$, we need to see the behavior of $M_{\eta}$ along the solutions of (\ref{MA2}). 
The following lemma asserts that $M_{\eta}$ is non-increasing along the solutions, whose proof can be given as in \cite{BM}. 

\begin{lem}\label{energy_inequality}
Let $\{\varphi_{t}\ |\ t \in [0, 1] \}$ be an arbitrary smooth family of solution of (\ref{MA2}). 
Then 
\begin{equation*}
\frac{d M_{\eta}(\varphi_{t})}{dt} 
= -(1-t)(2n + 2)\frac{d}{dt}\left( I_{\eta}(\varphi_{t}) - J_{\eta}(\varphi_{t})\right)
\leq 0. 
\end{equation*}
\end{lem}
Combining Proposition \ref{a priori}, Lemma \ref{energy_inequality} and Proposition \ref{functional_inequality}, we obtain the following result. 

\begin{thm}\label{sol}
Let $0< \tau <1$. 
Then any solution $\varphi_{\tau}$ of (\ref{MA2}) at $t = \tau$ uniquely extends to a smooth family $\{\varphi_{t}\ |\ t \in [0, \tau] \}$ of solutions of (\ref{MA2}). 
In particular the equation (\ref{MA2}) admits at most one solution at $t = \tau$. 
\end{thm}
In particular, if $\mathfrak{h}_{0} = \{ 0 \}$ then there exists at most one Sasaki-Einstein metric of $S$ which is compatible with $g$ by Remark \ref{openness2} and Theorem \ref{sol}. 

\begin{proof}[proof of Theorem \ref{sol}]
First note that a smooth family $\{\varphi_{t}\ |\ t \in [0, \tau] \}$ of solutions of (\ref{MA2}) is unique if it exists because of the implicit function theorem and the uniqueness of solutions of (\ref{MA2}) at $t = 0$. 
Hence it is sufficient to show that a solution $\varphi_{\tau}$ of (\ref{MA2}) at $t = \tau$ can be extended to a smooth family $\{\varphi_{t}\ |\ t \in [0, \tau] \}$ of solutions of (\ref{MA2}). 
We therefore assume, for contradiction, that any such extension is impossible. 
Then by Proposition \ref{openness} we have a maximal smooth family $\{ \varphi_{t}\ |\ t \in (\sigma, \tau]\}$ of solutions of (\ref{MA2}) for some $0 \leq \sigma$. 
In this proof we always denote by $t \in \mathbb{R}$ a real number satisfying $\sigma < t \leq \tau$. 
For arbitrary solution $\varphi_{t}$ we have 
\begin{equation*}
I_{\eta}(\varphi_{t}) 
\leq (n + 1) \left( I_{\eta}(\varphi_{t}) - J_{\eta}(\varphi_{t}) \right) 
\leq (n + 1) \left( I_{\eta}(\varphi_{\tau}) - J_{\eta}(\varphi_{\tau}) \right)  
\end{equation*}
by Lemma \ref{energy_inequality} and Proposition \ref{functional_inequality}. 
In particular, there exists a constant $A > 0$ which is independent of $t$ such that $I_{\eta}(\varphi_{t}) \leq A$. 
Hence by Lemma \ref{a priori}, there exists a constant $C > 0$ which depends only $A$, $n$ and the initial metric $g$ such that 
$\sup_{S}\varphi_{t} \leq C$. 
By El Kacimi-Alaoui's generalization of Yau's estimate, we can find a constant $C_{1} > 0$ such that $\| \varphi_{t} \|_{C^{2, \alpha}} \leq C_{1}$ for all $t \in (\sigma, \tau]$ and fixed $\alpha \in (0, 1)$. 
We now choose an arbitrary decreasing sequence $\{ t_{j} \}_{j=1}^{\infty} \subset (\sigma, \tau]$ such that $\lim_{j \to \infty}t_{j} = \sigma$. 
Then by Arzela-Ascoli's theorem, there exists a convergent subsequence of $\{ \varphi_{t_{j}} \}_{j=1}^{ \infty }$, which leads to a contradiction to the maximality of $\{ \varphi_{t}\ |\ t \in (\sigma, \tau]\}$. 
\qed \end{proof}
\subsection{Solutions at $t = 1$}
Next we mention at $t = 1$. 
By assumption, $\mathscr{E} \neq \phi$ and hence the equation (\ref{MA1}) has a solution at $t = 1$. 
We begin the following lemma. 
\begin{lem}\label{tangent_space}
Let $\{ \varphi_{t} \}_{t \in [0, 1]}$ be a smooth family of solutions of equation (\ref{MA1}). 
Put $\varphi := \varphi_{1}$ and $\eta_{SE} := \eta_{\varphi}$. 
Then 
\begin{equation}\label{kernel}
\int_{S}\varphi \psi (\frac{1}{2}d\eta_{SE})^{n} \wedge \eta_{SE} = 0
\end{equation}
for each $\psi \in \ker(\Box_{SE} - (2n+2))$, where $\Box_{SE}$ is the basic complex Laplacian for the Sasaki-Einstein metric $g_{SE} = g_{\varphi}$. 
\end{lem}
\begin{proof}
By differentiating the logarithms of both sides of equality (\ref{MA1}) at $t = 1$, we obtain 
\begin{equation*}
(\Box_{SE} - (2n+2))\dot \varphi_{t}|_{t = 1} = (2n + 2) \varphi. 
\end{equation*}
Then the lemma follows immediately. 
\qed \end{proof}

Consider the $G$-action on $\mathscr{E}$. Let $O$ be an arbitrary $G$-orbit in $\mathscr{E}$. 
For each $g_{SE} \in \mathscr{E}$, we can uniquely associate a function $\varphi = \varphi(g_{SE}) \in \mathscr{H}$ such that $g_{SE} = g_{\varphi}$ and $\varphi$ satisfies the equation (\ref{MA1}) at $t = 1$. 
Hence we can regard $O$ as a subset of the set of all solutions of the equation (\ref{MA1}) at $t = 1$ in $\mathscr{H}$. 
By the identification, we endow $O$ with the topology induced from the $C^{2, \alpha}$-norm on $C_{B}^{\infty}(S)$. 
Then the $G$-action on $O$ is clearly continuous. 
Hence the topology on $O$ coincides with the natural topology of the homogeneous space $O \cong G/K_{g_{SE}}$, where $K_{g_{SE}}$ is the isotropic subgroup of $G$ at $g_{SE}$. 
For each $\psi \in \ker(\Box_{SE} - (2n+2))$ we have associated normalized Hamiltonian holomorphic vector field $X_{\psi}$; 
\begin{equation*}
X_{\psi} = \psi \xi + \nabla^{i}\psi \frac{\partial}{\partial z^{i}} - \eta(\nabla^{i}\psi \frac{\partial}{\partial z^{i}}) \xi 
\end{equation*}
for a foliation coordinate $(x_{0}, z_{1}, \cdots, z_{n})$ (see Theorem 5.1 of \cite{FOW}). 
Let $f_{\psi, t}$ be a corresponding one-parameter group; $f_{\psi, t} = \exp(t X_{\psi}^{\mathbb{R}})$, where $X_{\psi}^{\mathbb{R}}$ is the real part of $X_{\psi}$. 
We put $g_{SE}(t) := f_{\psi, t}^{*}g_{SE}$ and $\varphi(t) := \varphi(g_{SE}(t))$. 
Then we can check easily that $\dot \varphi(0) = \psi + C$ for some $C \in \mathbb{R}$. 
On the other hand, since $\varphi(t)$ satisfies the equation (\ref{MA1}) we have $\Box_{SE}\dot \varphi(0) = (2n + 2)\dot \varphi(0)$ by differentiating the equality (\ref{MA1}). 
This shows that $C = 0$ and hence $\dot \varphi(0) = \psi$. 

Conversely, for each smooth curve $g(t) \in O$ with $g(0) = g_{SE}$, take the corresponding smooth functions $\varphi(t) \in \mathscr{H}$. 
Then we have $\Box_{SE}\dot \varphi(0) = (2n + 2)\dot \varphi(0)$ by differentiating the identity (\ref{MA1}). 
Thus we obtain 
\begin{equation*}
T_{g_{SE}}O \cong \ker(\Box_{SE} - (2n+2)). 
\end{equation*}

Define $\iota := (I_{\eta} - J_{\eta})|_{O}\ (\geq 0) : O \to \mathbb{R}$. 
The basic properties of $\iota$ are as follows. 
\begin{lem}\label{critical}
Let $g_{SE} \in O$. 
Then the followings are equivalent. 
\begin{enumerate}
\item $g_{SE}$ is a critical point of $\iota$, 
\item $\varphi(g_{SE})$ satisfies the condition (\ref{kernel}). 
\end{enumerate}
\end{lem}
This is immediately from (\ref{differential}). 
The following lemma shows the existence of a minimizer of $\iota$. 
\begin{lem}\label{properness}
The functional $\iota$ is proper. 
In particular, its minimum is always attained at some point of the orbit $O$. 
\end{lem}
\begin{proof}
Let $g_{SE} \in O$ with $|\iota(g_{SE})| \leq r$ for some $r > 0$. 
Then by Proposition \ref{functional_inequality} we have $I_{\eta}(\varphi) \leq (n+1)r$ for $\varphi := \varphi(g_{SE})$. 
Since $\rho_{\varphi}^{T} = (2n + 2)\omega_{\varphi}^{T}$, we have 
\begin{align*}
{\rm osc}_{S}\varphi 
&\leq C_{r}
\end{align*}
by Proposition \ref{osc}, where $K := \sup (-G_{\eta})$ and 
\begin{equation*}
C_{r} = (n + 1)r + 2n\left( KV_{0} + \gamma \frac{(2\pi)^{2}(2n-1)}{6n}\right). 
\end{equation*}
On the other hand, from (\ref{MA1}) we obtain 
\begin{equation*}
\int_{S}(\frac{1}{2}d\eta)^{n} \wedge \eta 
= \int_{S}(\frac{1}{2}d\eta_{\varphi})^{n} \wedge \eta_{\varphi} 
= \int_{S}\exp(-(2n + 2)\varphi + h)(\frac{1}{2}d\eta)^{n} \wedge \eta 
\end{equation*}
and hence 
$(2n + 2)\varphi(p) = h(p)$ for some $p \in S$. 
Hence we obtain 
\begin{align*}
\sup_{S}\varphi 
&= {\rm osc}_{S}\varphi - \inf_{S}\varphi 
\leq {\rm osc}_{S}\varphi - \varphi(p) \\
&= {\rm osc}_{S}\varphi - \frac{h(p)}{2n + 2} 
\leq {\rm osc}_{S}\varphi + \frac{1}{2n + 2}\sup_{S}|h| \\
&\leq C_{r} + \frac{1}{2n + 2}\sup_{S}|h|. 
\end{align*}
Thus if we put $C^{1}_{r} := C_{r} + \frac{1}{2n + 2}\sup_{S}|h|$ we have $\sup_{S}\varphi \leq C^{1}_{r}$. 
Then the result follows from El Kacimi-Alaoui's generalization of Yau's estimate \cite{Yau}. 
\qed \end{proof}
Then we shall calculate the Hessian of $\iota$ at a critical point. 
For the proof, we need the following formula for $g_{SE} \in \mathscr{E}$, which is shown by the same calculation as K${\rm \ddot a}$hler geometry (see \cite{BM} and Theorem 5.1 in \cite{FOW}); 
\begin{align}\label{Toshiki} 
\Box_{SE}\langle \partial_{B}\psi, \partial_{B} \varphi^{\prime} \rangle 
&= -\langle \partial_{B} \bar \partial_{B} \psi, \partial_{B} \bar \partial_{B} \varphi^{\prime} \rangle + \langle \partial_{B}(\Box_{SE}\psi), \partial_{B} \varphi^{\prime} \rangle, 
\end{align}
for any $\varphi^{\prime} \in \ker(\Box_{SE} - (2n + 2))$ and $\psi \in C_{B}^{\infty}(S)$, where $\langle \cdot, \cdot \rangle$ is the natural Hermitian pairings on basic forms induced from the transverse K${\rm \ddot a}$hler metric $g_{SE}^{T}$. 
\begin{lem}\label{Hessian}
Let $g_{SE} \in O$ be a critical point of $\iota$. 
Then the Hessian $({\rm Hess\ \iota})_{g_{SE}}$ of $\iota$ at $g_{SE}$ is given by 
\begin{equation*}
({\rm Hess\ \iota})_{g_{SE}}(\varphi^{\prime}, \varphi^{\prime \prime}) 
= \frac{2n + 2}{V}\int_{S}\left( 1 - \frac{1}{2}\Box_{SE}\varphi \right) \varphi^{\prime} \varphi^{\prime \prime} (\frac{1}{2}d\eta_{\varphi})^{n} \wedge \eta_{\varphi}
\end{equation*}
for each $\varphi^{\prime}, \varphi^{\prime \prime} \in \ker (\Box_{SE} - (2n + 2)) \cong T_{g_{SE}}O$, where $\varphi := \varphi(g_{SE})$. 
\end{lem}
\begin{proof}
Let $\{ \varphi_{s, t}\ |\ (-\varepsilon, \varepsilon) \times(-\varepsilon, \varepsilon)\}$ be a smooth family of functions  satisfying the following conditions; 
\begin{equation*}
g_{\varphi_{s, t}} \in O,\ 
\varphi_{0, 0} = \varphi,\ \frac{\partial \varphi_{s, t}}{\partial s} |_{s, t = 0} = \varphi^{\prime},\ \frac{\partial \varphi_{s, t}}{\partial t} |_{s, t = 0} = \varphi^{\prime \prime}. 
\end{equation*}
We shall denote $\Box_{\varphi_{s, t}}$ by $\Box_{s, t}$ for brevity. 
Since $\varphi_{s, t}$ satisfies (\ref{MA1}), we obtain 
\begin{equation}\label{1diff}
(-\Box_{s, t} + (2n + 2))\frac{\partial \varphi_{s, t}}{\partial t} = 0
\end{equation}
by differentiating the equation (\ref{MA1}) with respect to $t$. 
Further differentiation with respect to $s$ yields 
\begin{equation}\label{2diff}
-\left \langle \partial_{B}\bar \partial_{B}\frac{\partial \varphi_{s, t}}{\partial s}, \partial_{B}\bar \partial_{B}\frac{\partial \varphi_{s, t}}{\partial t} \right \rangle_{s, t} + (-\Box_{s, t} + (2n + 2))\left( \frac{\partial^{2} \varphi_{s, t}}{\partial s \partial t}\right) = 0, 
\end{equation}
where $\langle \cdot, \cdot \rangle_{s, t}$ is the natural Hermitian pairing on complex basic forms induced from the transverse K${\rm \ddot a}$hler metric $g_{\varphi_{s, t}}^{T}$. 
By evaluating this at $(s, t) = (0, 0)$, we obtain 
\begin{align*}
(-\Box_{SE} + (2n + 2))\left( \frac{\partial^{2} \varphi_{s, t}}{\partial s \partial t}\right)|_{s, t = 0} 
&= \left \langle \partial_{B}\bar \partial_{B}\varphi^{\prime}, \partial_{B}\bar \partial_{B}\varphi^{\prime \prime} \right \rangle_{0, 0} \\
&= -\Box_{SE}\langle \partial_{B}\varphi^{\prime}, \partial_{B}\varphi^{\prime \prime} \rangle_{0, 0} + \langle \partial_{B}\Box_{SE}\varphi^{\prime}, \partial_{B}\varphi^{\prime \prime} \rangle_{0, 0}\\
&= (-\Box_{SE} + (2n + 2))\langle \partial_{B}\varphi^{\prime}, \partial_{B}\varphi^{\prime \prime} \rangle_{0, 0}. 
\end{align*}
This shows that 
\begin{align*}
\frac{\partial^{2} \varphi_{s, t}}{\partial s \partial t}|_{s, t = 0} 
&\equiv \langle \partial_{B}\varphi^{\prime}, \partial_{B}\varphi^{\prime \prime} \rangle_{0, 0}\ \text{(mod $\ker(-\Box_{SE} + (2n + 2)) \otimes \mathbb{C}$)} \\
&\equiv \langle \partial_{B}\varphi^{\prime \prime}, \partial_{B}\varphi^{\prime} \rangle_{0, 0}\ \text{(mod $\ker(-\Box_{SE} + (2n + 2)) \otimes \mathbb{C}$)}. 
\end{align*}
Now we can calculate the Hessian of $\iota$; 
\begin{align*}
({\rm Hess}\ \iota)_{g_{SE}}(\varphi^{\prime}, \varphi^{\prime \prime}) 
&= \frac{\partial^{2}}{\partial s \partial t}\left( I_{\eta}(\varphi_{s, t}) - J_{\eta}(\varphi_{s, t}) \right)|_{s, t = 0} \\
&= \frac{\partial}{\partial s}\left\{ \frac{1}{V}\int_{S}\varphi_{s, t}\Box_{s,t}\frac{\partial \varphi_{s, t}}{\partial t}(\frac{1}{2}d\eta_{\varphi_{s, t}})^{n} \wedge \eta_{\varphi_{s, t}} \right\}|_{s, t = 0} \\
&= \frac{\partial}{\partial s}\left\{ \frac{2n + 2}{V}\int_{S}\varphi_{s, t}\frac{\partial \varphi_{s, t}}{\partial t}(\frac{1}{2}d\eta_{\varphi_{s, t}})^{n} \wedge \eta_{\varphi_{s, t}} \right\}|_{s, t = 0}\quad \text{(cf. (\ref{1diff}))} \\
&= \frac{2n + 2}{V}\int_{S} \varphi^{\prime} \varphi^{\prime \prime} \wedge (\frac{1}{2}d\eta_{\varphi})^{n} \wedge \eta_{\varphi} \\
&\quad + \frac{2n + 2}{V}\int_{S} \varphi \frac{\partial^{2} \varphi_{s, t}}{\partial s \partial t} |_{s, t = 0} \wedge (\frac{1}{2}d\eta_{\varphi})^{n} \wedge \eta_{\varphi} \\
&\quad - \frac{2n + 2}{V}\int_{S} \varphi \varphi^{\prime \prime}\Box_{SE}\varphi^{\prime} \wedge (\frac{1}{2}d\eta_{\varphi})^{n} \wedge \eta_{\varphi} \\
&= \frac{2n + 2}{V}\int_{S} \varphi^{\prime} \varphi^{\prime \prime} \wedge (\frac{1}{2}d\eta_{\varphi})^{n} \wedge \eta_{\varphi} \\
&\quad + \frac{2n + 2}{V}\int_{S} \varphi \frac{\langle \partial_{B} \varphi^{\prime}, \partial_{B}\varphi^{\prime \prime} \rangle_{0, 0} + \langle \partial_{B} \varphi^{\prime \prime}, \partial_{B}\varphi^{\prime} \rangle_{0, 0}}{2} \wedge (\frac{1}{2}d\eta_{\varphi})^{n} \wedge \eta_{\varphi} \\
&\quad - \frac{(2n + 2)^{2}}{V}\int_{S} \varphi \varphi^{\prime \prime}\varphi^{\prime} \wedge (\frac{1}{2}d\eta_{\varphi})^{n} \wedge \eta_{\varphi} \\
&= \frac{2n + 2}{V}\int_{S} \left\{ \varphi^{\prime} \varphi^{\prime \prime} 
- \frac{1}{2} \varphi \Box_{SE}(\varphi^{\prime} \varphi^{\prime \prime})  \right\} \wedge (\frac{1}{2}d\eta_{\varphi})^{n} \wedge \eta_{\varphi} \\
&= \frac{2n + 2}{V}\int_{S} \left( 1- \frac{1}{2}\Box_{SE}\varphi
\right) \varphi^{\prime} \varphi^{\prime \prime} (\frac{1}{2}d\eta_{\varphi})^{n} \wedge \eta_{\varphi}. 
\end{align*}
\qed \end{proof}

The following proposition is crucial for our proof of Theorem B. 

\begin{prop}\label{continuation}
For every critical point $g_{SE} \in O$ of $\iota$ with non-degenerate Hessian, $\varphi_{1} := \varphi(g_{SE})$ can be extended to a smooth family $\{ \varphi_{t}\ |\ t \in [1-\varepsilon, 1]\}$ of solutions of (\ref{MA1}) for some $\varepsilon > 0$. 
\end{prop}

\begin{proof}
Put $L_{B}^{2}(S; g_{SE})$ to be the closure of $C_{B}^{\infty}(S)$ in $L^{2}(S; g_{SE})$. 
Let $W := \ker(\Box_{\varphi_{1}} - (2n + 2))$ and $P$ be the orthogonal projection from $L_{B}^{2}(S; g_{SE})$ to $W$. 
Fixing $\alpha \in \mathbb{R}$ with $0 < \alpha < 1$, define $W_{k, \alpha}^{\perp}$ to be the intersection of the orthogonal complement of $W$ and $C_{B}^{k, \alpha}(S)$, for $k = 0, 1, 2, \cdots$. 
Recall that $\varphi_{1}$ belongs to $W_{k, \alpha}^{\perp}$. 
Let $k \geq 2$, and consider the mapping 
\begin{equation*}
\Psi: \mathbb{R} \times \mathscr{H}^{k, \alpha}(S) \to C_{B}^{k-2, \alpha}(S),\ \Psi(t, \varphi) := \log \left( \frac{\det(g^{T}_{i \bar{j}} + \frac{\partial^{2}\varphi}{\partial z^{i}\partial \bar z^{j}})}{\det(g^{T}_{i \bar{j}})} \right) + t(2n+2)\varphi - h. 
\end{equation*}
Note that, by the well-known regularity theorem, any $\varphi \in \mathscr{H}^{k, \alpha}$ satisfying $\Psi(t, \varphi) = 0$ is automatically smooth. 
For each $\varphi \in \mathscr{H}^{k, \alpha}(S)$, we write 
\begin{equation*}
\varphi = \varphi_{1} + \psi + \theta, 
\end{equation*}
where $\psi := P(\varphi - \varphi_{1}) \in W$ and $\theta := (1 - P)(\varphi - \varphi_{1}) \in W_{k, \alpha}^{\perp}$. 
Now the equation 
\begin{equation}\label{MA_last}
\Psi(t, \varphi) = 0
\end{equation}
is written in the form 
\begin{equation*}
P \Psi(t, \varphi_{1} + \psi + \theta) = 0,\quad \Psi_{0}(t, \psi, \theta) = 0, 
\end{equation*}
where $\Psi_{0}$ is defined by 
\begin{equation*}
\Psi_{0}(t, \psi, \theta) := (1-P) \Psi(t, \varphi_{1} + \psi + \theta). 
\end{equation*}
Then clearly $\Psi_{0}(1, 0, 0) = 0$ and the Fr${\rm \acute e}$chet derivative $D_{\theta}\Psi_{0}|_{(1, 0, 0)}$ of $\Psi_{0}$ with respect to $\theta$ at $(t, \psi, \theta) = (1, 0, 0)$ is 
\begin{equation*}
D_{\theta}\Psi_{0}|_{(1, 0, 0)}(\theta^{\prime}) = (-\Box_{SE} + (2n + 2))\theta^{\prime}, 
\end{equation*}
which is invertible. 
Hence by the implicit function theorem we obtain a smooth mapping $U \ni (t, \psi) \mapsto \theta_{t, \psi} \in W_{k, \alpha}^{\perp}$ of a neighborhood $U$ of $(1, 0) \in \mathbb{R} \times W$ to $W_{k, \alpha}^{\perp}$ such that 
\begin{enumerate}
\item $\theta_{1, 0} = 0$, 
\item $\|\theta_{t, \psi} \|_{C^{k, \alpha}} \leq \delta$ on $U$ for some $\delta > 0$ and 
\item $\Psi_{0}(t, \psi, \theta) = 0$ (where $\|\theta \|_{C^{k, \alpha}} \leq \delta$) is, as an equation in $\theta \in C_{B}^{k, \alpha}(S)$, uniquely solvable in the form $\theta = \theta_{t, \psi}$ on $U$. 
\end{enumerate}
By differentiating the identity $\Psi_{0}(t, \psi, \theta_{t, \psi}) = 0$ at $(1, 0)$ we obtain 
\begin{align}
&(-\Box_{SE} + (2n + 2)) \left( \frac{\partial}{\partial t}\theta_{t, \psi} |_{(1, 0)} \right) = -(2n + 2) \varphi_{1}, \label{identity1} \\
&(D_{\psi}\theta_{t, \psi})|_{(1, 0)}(\psi^{\prime}) = 0\quad \text{for all $\psi^{\prime} \in W$}. \label{identity2}
\end{align}
Then the equation (\ref{MA_last}), on a small neighborhood of $\varphi_{1}$, reduces to 
\begin{equation*}
\Psi_{1}(t, \psi) = 0, 
\end{equation*}
where $\Psi_{1}(t, \psi) := P\Psi(t, \varphi_{1} + \psi + \theta_{t, \psi})$ for $(t, \psi) \in U$. 
Recall that $\Psi(1, \varphi) = 0$ for all $\varphi \in O$. 
Hence $\Psi_{1} = 0$ on $\{ t = 1 \}$ and therefore the mapping 
\begin{equation*}
U|_{t \neq 1} \ni (t, \psi) \mapsto \Psi_{2}(t, \psi) := \frac{\Psi_{1}(t, \psi)}{t - 1}
\end{equation*}
naturally extends to a smooth map on $U$ to $W$ (denoted by the same $\Psi_{2}$). 
Note that, for $t = 1$, we have 
\begin{equation*}
\Psi_{2}(1, 0) = \frac{\partial \Psi_{1}(t, 0)}{\partial t}|_{t=1} = 0. 
\end{equation*}
Hence, if the Fr${\rm \acute e}$chet derivative $D_{\psi} \Psi_{2}|_{(1, 0)}$ is invertible, we obtain the desired result. 
The Fr${\rm \acute e}$chet derivative $D_{\psi} \Psi_{2}|_{(1, 0)}$ is written in the following form, whose proof is given later.  
\begin{lem}\label{Hessian2}
For each $\psi^{\prime}, \psi^{\prime \prime} \in W$, 
\begin{align*}
\int_{S} D_{\psi}\Psi_{2} |_{(1, 0)}(\psi^{\prime}) \cdot \psi^{\prime \prime} ( \frac{1}{2}d\eta_{1})^{n} \wedge \eta_{1} 
&= (2n + 2)\int_{S}\left( 1 - \frac{1}{2}\Box_{SE} \right)\psi^{\prime} \psi^{\prime \prime}(\frac{1}{2}d\eta_{1})^{n} \wedge \eta_{1} \\
&= V({\rm Hess}\ \iota)_{g_{SE}}(\psi^{\prime}, \psi^{\prime \prime}). 
\end{align*}
\end{lem}
Then by this lemma, $D_{\psi} \Psi_{2}|_{(1, 0)}$ is invertible. 
Hence the implicit function theorem shows that the equation $\Psi_{2}(t, \psi) = 0$ in $\psi$ is uniquely solvable in a neighborhood of $(1, 0)$ to produce a smooth curve $\{\psi(t)\ |\ t \in (1-\varepsilon, 1] \}$ in $\ker(\Box_{SE} - (2n + 2))$ such that $\psi(1) = 0$ and $\Psi_{2}(t, \psi(t)) = 0$. 
Therefore, we have $\Psi(t, \varphi_{1} + \psi(t) + \theta_{t, \psi(t)}) = 0$ for $t \in (1-\varepsilon, 1]$ and hence $\{\varphi_{1} + \psi(t) + \theta_{t, \psi(t)}\ |\ t \in (1-\varepsilon, 1]\}$ is a one parameter family of solutions of (\ref{MA1}). 

Finally, we shall prove Lemma \ref{Hessian2}. 
First we shall show the following formula;
\begin{align}\label{integral}
&-\int_{S}\varphi^{\prime} \langle \partial_{B} \bar \partial_{B} \psi, \partial_{B} \bar \partial_{B} \varphi^{\prime \prime} \rangle_{\varphi_{1}} (\frac{1}{2}d\eta_{1})^{n} \wedge \eta_{1} \\
&\quad = \int_{S}(2n + 2)\varphi^{\prime} \varphi^{\prime \prime} - \langle \partial_{B} \varphi^{\prime}, \partial_{B} \varphi^{\prime \prime} \rangle_{\varphi_{1}} ((-\Box_{\varphi_{1}} + (2n + 2))\psi) (\frac{1}{2}d\eta_{1})^{n} \wedge \eta_{1} \notag
\end{align}
 for each $\varphi^{\prime}, \varphi^{\prime \prime} \in W$ and $\psi \in C_{B}^{\infty}(S)$. 
For (\ref{integral}), put $\zeta := (-\Box_{\varphi_{1}} + (2n + 2))\psi$. 
Then we have 
\begin{align*}
&\int_{S}(2n + 2) \varphi^{\prime} \varphi^{\prime \prime} - \langle \partial_{B} \varphi^{\prime}, \partial_{B} \varphi^{\prime \prime} \rangle_{\varphi_{1}} \zeta (\frac{1}{2}d\eta_{1})^{n} \wedge \eta_{1} \\
&\quad = -\sqrt{-1}\int_{S}n \zeta (\varphi^{\prime} \partial_{B} \bar \partial_{B} \varphi^{\prime \prime} + \partial_{B} \varphi^{\prime} \wedge \bar \partial_{B} \varphi^{\prime \prime}) \wedge (\frac{1}{2}d\eta_{1})^{n-1} \wedge \eta_{1} \\
&\quad = -\sqrt{-1}\int_{S}n\zeta \partial_{B}(\varphi^{\prime} \bar \partial_{B} \varphi^{\prime \prime}) \wedge (\frac{1}{2}d\eta_{1})^{n-1} \wedge \eta_{1} \\
&\quad = \sqrt{-1}\int_{S}n\varphi^{\prime}\partial_{B}\zeta \wedge \bar \partial_{B} \varphi^{\prime \prime} \wedge (\frac{1}{2}d\eta_{1})^{n-1} \wedge \eta_{1} \\
&\quad = \int_{S}\varphi^{\prime} \langle \partial_{B}\zeta, \partial_{B} \varphi^{\prime \prime} \rangle_{\varphi_{1}} (\frac{1}{2}d\eta_{1})^{n} \wedge \eta_{1} \\
&\quad = \int_{S}\varphi^{\prime} \langle \partial_{B}(-\Box_{\varphi_{1}} + (2n + 2))\psi, \partial_{B} \varphi^{\prime \prime} \rangle_{\varphi_{1}} (\frac{1}{2}d\eta_{1})^{n} \wedge \eta_{1} \\
&\quad = \int_{S}\varphi^{\prime} \left\{  -\Box_{\varphi_{1}}\langle \partial_{B}\psi, \partial_{B} \varphi^{\prime \prime} \rangle_{\varphi_{1}} + (2n + 2)\langle \partial_{B} \psi, \partial_{B} \varphi^{\prime \prime} \rangle_{\varphi_{1}} \right\}(\frac{1}{2}d\eta_{1})^{n} \wedge \eta_{1} \\
&\quad \quad -\int_{S}\varphi^{\prime}\langle \partial_{B} \bar \partial_{B} \psi, \partial_{B} \bar \partial_{B} \varphi^{\prime \prime} \rangle_{\varphi_{1}} (\frac{1}{2}d\eta_{1})^{n} \wedge \eta_{1} \\
&\quad = -\int_{S}\varphi^{\prime}\langle \partial_{B} \bar \partial_{B} \psi, \partial_{B} \bar \partial_{B} \varphi^{\prime \prime} \rangle_{\varphi_{1}} (\frac{1}{2}d\eta_{1})^{n} \wedge \eta_{1}. 
\end{align*}
This shows (\ref{integral}). 
Then, by (\ref{identity2}) we have 
\begin{align*}
D_{\psi}\Psi_{2}|_{(1, 0)}(\psi^{\prime}) 
&= D_{\psi} \frac{\partial \Psi_{1}}{\partial t}|_{(1, 0)}(\psi^{\prime}) \\
&= (2n + 2) \psi^{\prime} - P \left \langle \partial_{B} \bar \partial_{B} \left( \frac{\partial \theta_{t, \psi}}{\partial t} |_{(1, 0)}\right), \partial_{B} \bar \partial_{B}\psi^{\prime} \right \rangle_{\varphi_{1}} 
\end{align*}
for each $\psi^{\prime} \in W$. 
Hence, it follows that 
\begin{align*}
&\int_{S} D_{\psi}\Psi_{2} |_{(1, 0)}(\psi^{\prime}) \cdot \psi^{\prime \prime} ( \frac{1}{2}d\eta_{1})^{n} \wedge \eta_{1} \\
&\quad= \int_{S}\left\{ (2n + 2) \psi^{\prime}\psi^{\prime \prime} - \psi^{\prime \prime} \left \langle \partial_{B} \bar \partial_{B} \left( \frac{\partial \theta_{t, \psi}}{\partial t} |_{(1, 0)}\right), \partial_{B} \bar \partial_{B}\psi^{\prime} \right \rangle_{\varphi_{1}} \right\}(\frac{1}{2}d\eta_{1})^{n} \wedge \eta_{1} \\
&\quad= (2n + 2)\int_{S}\left\{ \psi^{\prime}\psi^{\prime \prime} - ( (2n + 2)\psi^{\prime}\psi^{\prime \prime} - \langle \partial_{B}\psi^{\prime}, \partial_{B}\psi^{\prime \prime} \rangle_{\varphi_{1}} )\varphi \right\}(\frac{1}{2}d\eta_{1})^{n} \wedge \eta_{1} \\
&\quad= (2n + 2)\int_{S}\left( 1 - \frac{1}{2}\Box_{SE} \varphi \right)\psi^{\prime}\psi^{\prime \prime}(\frac{1}{2}d\eta_{1})^{n} \wedge \eta_{1} \\
&\quad = V({\rm Hess}\ \iota)_{g_{SE}}(\psi^{\prime}, \psi^{\prime \prime}). 
\end{align*}
This proves the lemma. 
\qed \end{proof}

\begin{rem}\label{gauge}
Fix a $G$-orbit $O$ in $\mathscr{E}$ arbitrary and take a minimizer $g_{SE}$ of $\iota : O \to \mathbb{R}$. 
Then $g_{SE}$ is a critical point of $\iota$ and the Hessian is automatically positive semi-definite. 
We shall realize a critical point for $\iota$ with positive definite Hessian by a small change of the initial metric $g$. 

For sufficient small $\delta \in (0, 1)$, define $g^{\delta} := g_{\delta\varphi_{1}}$, where $\varphi_{1} = \varphi(g_{SE})$. 
The associated transverse K${\rm \ddot a}$hler form is given by $(\omega^{\delta})^{T} = (1-\delta)\omega^{T} + \delta \omega_{1}^{T} = \omega^{T} + \delta \sqrt{-1}\partial_{B}\bar \partial_{B}\varphi_{1}$. 
When the role of the initial metric $g$ is played by the new Sasaki metric $g^{\delta}$, the Sasaki-Einstein metric $g_{SE} = g_{\varphi_{1}}$ corresponds to $g_{\varphi_{1}^{\delta}}^{\delta}$ for a basic function 
\begin{equation*}
\varphi_{1}^{\delta} = (1 - \delta)\varphi_{1} + C_{\delta}, 
\end{equation*}
where $C_{\delta}$ is a constant. 
Then $g_{SE}$ is a critical point of $\iota^{\delta}$ with positive definite Hessian, where $\iota^{\delta}$ denotes the one corresponding to $\iota$. 
Indeed, for each $\psi \in \ker(\Box_{g_{SE}} - (2n + 2))$ we have 
\begin{equation*}
\int_{S}\psi \varphi_{1}^{\delta} (\frac{1}{2}d\eta_{SE})^{n} \wedge \eta_{SE} = (1-\delta) \int_{S}\psi \varphi_{1}(\frac{1}{2}d\eta_{SE})^{n} \wedge \eta_{SE} = 0 
\end{equation*}
and hence $\varphi_{1}^{\delta}$ is a critical point of $\iota^{\delta}$ by Lemma \ref{critical}. 
Moreover, by Lemma \ref{Hessian} we have 
\begin{align*}
({\rm Hess}\ \iota^{\delta})_{g_{SE}}(\psi^{\prime}, \psi^{\prime \prime}) 
&= \frac{2n + 2}{V}\int_{S}\left( 1 - \frac{1}{2}\Box_{SE}\varphi_{1}^{\delta} \right)\psi^{\prime} \psi^{\prime \prime}(\frac{1}{2}d\eta_{SE})^{n} \wedge \eta_{SE} \\
&= (1 - \delta)({\rm Hess}\ \iota)_{g_{SE}}(\psi^{\prime}, \psi^{\prime \prime}) \\
&\quad + \delta \frac{2n + 2}{V}\int_{S}\psi^{\prime} \psi^{\prime \prime}(\frac{1}{2}d\eta_{SE})^{n} \wedge \eta_{SE}, 
\end{align*}
where $\Box_{SE}$ is the basic complex Laplacian with respect to the Sasaki metric $g_{SE}$. 
This shows that $({\rm Hess}\ \iota^{\delta})_{g_{SE}}$ is positive definite. 
Hence by the argument in the last subsection and Proposition \ref{continuation}, $\varphi_{1}^{\delta}$ can be uniquely extended to a smooth family $\{ \varphi_{t}^{\delta}\ |\ t \in [1-\varepsilon, 1] \}$ of the equation (\ref{MA1}) with respect to the initial metric $g^{\delta}$. 
\end{rem}

\subsection{Proof of Theorem B}
Let $O^{\prime}$ and $O^{\prime \prime}$ be arbitrary $G$-orbits in $\mathscr{E}$. 
Then by the argument of Remark \ref{gauge}, for a suitable choice of the initial metric $g_{0}$, the function $\iota^{\prime}: O^{\prime} \ni g^{\prime} \mapsto I(g_{0}, g^{\prime}) - J(g_{0}, g^{\prime}) \in \mathbb{R}$ has a critical point $g_{SE}^{\prime} \in O^{\prime}$ with positive definite Hessian. 
Recall that the function $\iota^{\prime \prime}: O^{\prime \prime} \ni g^{\prime \prime} \mapsto I(g_{0}, g^{\prime \prime}) - J(g_{0}, g^{\prime \prime}) \in \mathbb{R}$ takes its minimum at some point $g_{SE}^{\prime \prime} \in O^{\prime \prime}$. 
We now put $g_{0}^{\delta} := (1 - \delta) g_{0} + \delta g_{SE}^{\prime \prime}$ for $\delta \in [0, 1]$. 
Again by the argument of Remark \ref{gauge} applied to $O^{\prime \prime}$, $g_{SE}^{\prime \prime}$ is a critical point of $(\iota^{\prime \prime})^{\delta}$ with positive definite Hessian whenever $\delta \in (0, 1]$. 
Hence by Proposition \ref{continuation}, $\varphi_{1}^{\prime \prime} := \varphi(g_{SE}^{\prime \prime})$ (with respect to the initial metric $g_{0}^{\delta}$) can be extended uniquely to a smooth family $\{\varphi_{t}^{\prime \prime}\ |\ t \in (1-\varepsilon, 1] \}$ of (\ref{MA1}) with respect to the initial metric $g_{0}^{\delta}$. 

We finally consider the functional $\iota_{\delta}^{\prime}: O^{\prime} \ni g^{\prime} \mapsto I(g_{0}^{\delta}, g^{\prime}) - J(g_{0}^{\delta}, g^{\prime}) \in \mathbb{R}$. 
Note that $\iota_{\delta}^{\prime}$ converges to $\iota^{\prime}$ as $\delta$ tends to $0$. 
Then for a sufficiently small $\delta > 0$, $g_{SE}^{\prime}$ is a critical point of $\iota_{\delta}^{\prime}$ with positive definite Hessian. 
Hence $\varphi_{1}^{\prime} := \varphi(g_{SE}^{\prime})$ (with respect to the initial metric $g_{0}^{\delta}$) can be extended uniquely to a smooth family $\{\varphi_{t}^{\prime}\ |\ t \in (1-\varepsilon, 1] \}$ of (\ref{MA1}) with respect to the initial metric $g_{0}^{\delta}$. 
By Theorem \ref{sol} and the equivalence between the equations (\ref{MA1}) and (\ref{MA2}) for $t \in (0, 1]$, we conclude $g^{\prime} = g^{\prime \prime}$. 
Thus, $O^{\prime} = O^{\prime \prime}$ and the proof is now complete. 

\section{Concluding remarks}
Theorem A plays a central role to obtain a priori $C^{0}$-estimate for solutions of the equation (\ref{MA2}) (cf. Proposition \ref{osc}). 
We remark that a similar estimate can be obtained without the diameter bound in the following way. 
Let $(S, g)$ be a $(2n + 1)$-dimensional Sasaki manifold with Sasakian structure $\mathcal{S} = \{g, \xi, \eta, \Phi \}$. 
Consider a solution $\varphi_{t} \in \mathscr{H}$ of the equation (\ref{MA2}) at $t$. 
Note that, as shown in Remark \ref{remark_positive}, $\varphi_{t}$ satisfies 
\begin{equation*}
\rho^{T}_{\varphi_{t}} \geq t(2n + 2)\omega^{T}_{\varphi_{t}}. 
\end{equation*}
We introduce a family of contact structures by multiplication of positive constant $\mu$, 
\begin{align}
\eta_{\varphi_{t}, \mu} &= \mu^{-1}\eta_{\varphi_{t}}, \label{D_contact}\\
\xi_{\mu} &= \mu \xi. \label{D_Reeb}
\end{align}
Then we see that $(\eta_{\varphi_{t}, \mu}, \xi_{\mu})$ gives a Sasakian structure with the metric $g_{\varphi_{t}, \mu}$ on $S$. 
The transversal metric $g^{T}_{\varphi_{t}, \mu}$ is given by $g^{T}_{\varphi_{t}, \mu} = \mu^{-1}g^{T}_{\varphi_{t}}$, and the volume form of $g_{\varphi_{t}, \mu}$ is given by 
\begin{equation}\label{volume}
\eta_{\varphi, \mu} \wedge (d\eta_{\varphi_{t}, \mu})^{n} = \mu^{-(n+1)}\eta_{\varphi_{t}} \wedge (d\eta_{\varphi_{t}})^{n}. 
\end{equation}
Let $\Delta_{\varphi_{t}, \mu}$ be the Laplacian with the Green function $G_{\varphi_{t}, \mu}$ and $Ric_{\varphi_{t}, \mu}$ the Ricci tensor with respect to $g_{\varphi_{t}, \mu}$. 
The following is a well-known fact on the Green function of compact Riemannian manifolds. 
\begin{fact}[\cite{BM}]\label{Gallot}
Let $(S, g)$ be a $(2n + 1)$-dimensional compact Riemannian manifold with the Green function $G(p, q)$. 
We assume 
\begin{equation*}
{\rm diam}(S, g)^{2} Ric \geq -\varepsilon^{2}g 
\end{equation*}
for a constant $\varepsilon \geq 0$. 
Then there exists a constant $\gamma(n, \varepsilon) > 0$ which depends on only $m$ and $\varepsilon$ and we have 
\begin{equation*}
G(p, q) \geq -\gamma(n, \varepsilon) \frac{{\rm diam}(S, g)^{2}}{{\rm Vol}(S, g)}
\end{equation*}
for the Green function of $(S, g)$. 
\end{fact}

Fact \ref{Gallot} has the following implication on the volume and the diameter of $(S, g_{\varphi_{t}, \mu})$. 

\begin{prop}\label{vol_diam}
Let $(S, g)$ be a $(2n + 1)$-compact Sasakian manifold and $\varphi_{t}$ a solution of (\ref{MA2}) at $t$. 
If we set $\mu = t^{-1}$, then we have estimates of the volume and the diameter with respect to the metric $g_{\varphi_{t}, \mu}$, 
\begin{align*}
{\rm Vol}(S, g_{\varphi_{t}, \mu}) &= t^{n+1}V_{0}, \\
{\rm diam}(S, g_{\varphi_{t}, \mu}) &\leq \pi, 
\end{align*}
where $V_{0} = \int_{S}\frac{1}{n!}(d\eta)^{n} \wedge \eta$. 
\end{prop}
\begin{proof}
Since $\mu = t^{-1}$, we have 
\begin{align*}
{\rm Vol}(S, g_{\varphi_{t}, \mu}) 
&= \int_{S}\frac{1}{n!}(d\eta_{\varphi_{t}, \mu})^{n} \wedge \eta_{\varphi_{t}, \mu} = \mu^{-(n+1)} \int_{S} \frac{1}{n!}(d\eta_{\varphi_{t}})^{n} \wedge \eta_{\varphi_{t}} \\
&= t^{n+1} \int_{S}\frac{1}{n!}(d\eta)^{n} \wedge \eta \\
&= t^{n+1}V_{0}. 
\end{align*}
Furthermore, we have 
\begin{equation*}
Ric_{\varphi_{t}, \mu}(X, Y) = Ric^{T}_{\varphi_{t}, \mu}(X, Y) - 2g_{\varphi_{t}, \mu}(X, Y) 
\end{equation*}
and 
\begin{equation*}
\mu g_{\varphi_{t}, \mu}(X, Y) = g_{\varphi_{t}}(X, Y) 
\end{equation*}
for all $X, Y \in \ker \eta_{\varphi_{t}, \mu}$. 
Since the transversal Ricci curvature is invariant under the multiplication by positive constant of a transversal metric, thus $Ric^{T}_{\varphi_{t}, \mu} = Ric^{T}_{\varphi_{t}}$, for all $X, Y \in \ker \eta_{\varphi_{t}, \mu}$. 
Then we have 
\begin{align*}
Ric^{T}_{\varphi_{t}, \mu}(X, Y) 
&= Ric^{T}_{\varphi_{t}}(X, Y) \\
&\geq t(2n + 2)g^{T}_{\varphi_{t}}(X, Y) \\
&= t(2n + 2) \mu g^{T}_{\varphi_{t}, \mu}(X, Y) \\
&= t(2n + 2) \mu g_{\varphi_{t}, \mu}(X, Y). 
\end{align*}
Therefore we have
\begin{equation*}
Ric_{\varphi_{t}, \mu}(X, Y) \geq t(2n + 2) \mu g_{\varphi_{t}, \mu}(X, Y) - 2g_{\varphi_{t}, \mu}(X, Y). 
\end{equation*}
Since $\mu = t^{-1}$, we have $Ric_{\varphi_{t}, \mu}(X, Y) \geq 2n g_{\varphi_{t}, \mu}(X, Y)$. 
It follows that 
\begin{align*}
Ric_{\varphi_{t}, \mu}(X, \xi_{\mu}) 
&= 2n\eta_{\varphi_{t}, \mu}(X) \\
&= 2ng_{\varphi_{t}, \mu}(X, \xi_{\mu}) 
\end{align*} 
for all $X \in TS$. 
Therefore we obtain 
\begin{equation*}
Ric_{\varphi_{t}, \mu} \geq 2ng_{\varphi_{t}, \mu}. 
\end{equation*}
Finally, we have ${\rm diam}(S, g_{\varphi_{t}, \mu}) \leq \pi$ by Myers' theorem. 
\end{proof}

Now we consider about the oscillation ${\rm osc}_{S}\varphi_{t}$ of $\varphi_{t}$. 
As shown in Proposition \ref{osc}, we have 
\begin{equation*}
\sup_{S}\varphi_{t} \leq \frac{1}{V}\int_{S}\varphi_{t}(d\eta)^{n} \wedge \eta - 2n K V_{0}, 
\end{equation*}
where $-K$ is the infimum of the Green function $G$ with respect to the metric $g$. 
We shall give an estimate for the infimum of $\varphi_{t}$. 
Let $\Delta_{t,\mu}$ be the Laplacian and $G_{t,\mu}$ the Green function with respect to the Sasaki metric $g_{\varphi_{t}, \mu}$. By Proposition \ref{vol_diam} and Fact \ref{Gallot}, we have 
\begin{equation}\label{sekiya_green}
G_{t, \mu} \geq -\gamma(n, 0)\frac{\pi^{2}}{t^{n+1}V}, 
\end{equation}
where $\mu = t^{-1}$ as in Proposition \ref{vol_diam}. 
We denote by $\Box_{t, \mu}$ the basic complex Laplacian with respect to the transversal K${\rm \ddot a}$hler form $\frac{1}{2}d\eta_{\varphi_{t}, \mu}$. 
Then it follows that $\Delta_{t,\mu}\varphi_{t} = 2\Box_{t, \mu}\varphi_{t}$. 
By $d\eta_{\varphi_{t}, \mu} = \mu^{-1}(d\eta + \sqrt{-1}\partial_{B} \bar{\partial}_{B} \varphi_{t})$, we have 
\begin{equation}\label{sekiya_laplacian}
\Box_{t,\mu}\varphi_{t} = \mu \Box_{t, \mu}\mu^{-1}\varphi_{t} = \mu {\rm tr}_{d\eta_{\varphi_{t}, \mu}}(d\eta_{\mu} - d\eta_{\varphi_{t}, \mu}) \geq -nt^{-1}, 
\end{equation}
where $\eta_{\mu} = \mu^{-1}\eta$ and $\mu = t^{-1}$. 
By applying the Fact \ref{Gallot} to $(S, g_{\varphi_{t}, \mu})$, we have 
\begin{align*}
\varphi_{t}(p) 
&= \frac{1}{t^{n+1}V}\int_{S}\varphi_{t}(d\eta_{\varphi_{t}, \mu})^{n} \wedge \eta_{\varphi_{t}, \mu} \\
&\quad + \int_{S} \left( G_{t, \mu}(p, q) + \gamma(n, 0)\frac{\pi^{2}}{t^{n+1}V} \right)(\Delta_{t, \mu}\varphi_{t})dV_{g_{\varphi_{t}, \mu}}(q). 
\end{align*}
By (\ref{volume}), the first term is given by 
\begin{align*}
\frac{1}{t^{n+1}V}\int_{S}\varphi_{t}(d\eta_{\varphi_{t}, \mu})^{n} \wedge \eta_{\varphi_{t}, \mu} 
= \frac{1}{V}\int_{S}\varphi_{t}(d\eta_{\varphi_{t}})^{n} \wedge \eta. 
\end{align*}
By using (\ref{sekiya_green}) and (\ref{sekiya_laplacian}), we have 
\begin{align*}
&\quad \int_{S} \left( G_{t, \mu}(p, q) + \gamma(n, 0) \frac{\pi^{2}}{t^{n+1}V} \right) (\Delta_{t, \mu}\varphi_{t})dV_{g_{\varphi_{t}, \mu}}(q) \\
&= \int_{S} \left( G_{t, \mu}(p, q) + \gamma(n, 0) \frac{\pi^{2}}{t^{n+1}V} \right) (2\Box_{t, \mu}\varphi_{t})dV_{g_{\varphi_{t}, \mu}}(q) \\
&\geq -\frac{2n}{t}\gamma(n, 0)\frac{\pi^{2}}{t^{n+1}V}\frac{t^{n+1}V}{n!} \\
&= -2n\gamma(n, 0)\frac{\pi^{2}}{t(n!)}. 
\end{align*}
Thus we obtain
\begin{equation}\label{sekiya_lower_bound}
\varphi_{t}(x) \geq \frac{1}{V}\int_{S}\varphi_{t}(d\eta_{\varphi_{t}})^{n} \wedge \eta - 2n \gamma(n, 0)\frac{\pi^{2}}{t(n!)}. 
\end{equation}
This gives the desired estimate 
\begin{align}\label{sekiya_osc}
{\rm osc}_{S}\varphi_{t} 
&= \sup_{S}\varphi_{t} - \inf_{S}\varphi_{t} \notag \\
&\leq I(0, \varphi_{t}) + 2n\left( KV_{0} + \gamma(n, 0)\frac{\pi^{2}}{t(n!)} \right). 
\end{align}
By applying the inequality (\ref{sekiya_osc}), we can prove directly the uniqueness of Sasaki-Einstein metrics up to the action of the identity component of the automorphism group for the transverse holomorphic structure. 

The deformation of a Sasakian structure defined by (\ref{D_contact}) and (\ref{D_Reeb}) is called a {\it $D$-homothetic deformation}. 
By applying $D$-homothetic deformations to complete Sasaki manifolds with positive transverse Ricci curvature, Hasegawa and Seino shows in \cite{HS} that such Sasaki manifolds are compact with finite fundamental group. 
Although this method does not lead to a diameter bound, it is applicable to $C^{0}$-estimates for solutions of (\ref{MA2}). 

\appendix
\section{Some functionals on the space of Sasakian metrics}
In this appendix, we introduce some functionals on the space of K${\rm \ddot a}$hler potentials for the transverse K${\rm \ddot a}$hler structure. 
Let $(S, g)$ be a $(2n + 1)$-dimensional Sasaki manifold with Sasakian structure $\mathcal{S} = (g, \xi, \eta, \Phi)$. We assume that $c_{1}^{B}(S) > 0$ and $c_{1}(D) = 0$. 
Let $V_{0} = \frac{1}{n!}\int_{S}(\frac{1}{2}d\eta)^{n} \wedge \eta$ be the volume of $(S, g)$ and put $V := n! V_{0}$. 

\subsection{Functionals $L_{\eta}$ and $M_{\eta}$}
For each $\varphi^{\prime}, \varphi^{\prime \prime} \in \mathscr{H}$, we put 
\begin{align*}
L(\varphi^{\prime}, \varphi^{\prime \prime}) 
&:= \frac{1}{V}\int_{a}^{b}dt\int_{S}\dot{\varphi_{t}}(\frac{1}{2}d\eta_{\varphi_{t}})^{n} \wedge \eta_{\varphi_{t}} , \\
M(\varphi^{\prime}, \varphi^{\prime \prime}) 
&:= -\frac{1}{V}\int_{a}^{b}dt\int_{S}\dot{\varphi_{t}}(s^{T}(\varphi_{t}) - n(2n + 2))(\frac{1}{2}d\eta_{\varphi_{t}})^{n} \wedge \eta_{\varphi_{t}}, 
\end{align*}
where $\{\varphi_{t}\ |\ t \in [a, b]\}$ is an arbitrary piecewise smooth path in $\mathscr{H}$ such that $\varphi_{a} = \varphi^{\prime}$ and $\varphi_{b} = \varphi^{\prime \prime}$ and $s^{T}(\varphi_{t})$ is the transverse scalar curvature for the transverse K${\rm \ddot a}$hler metric $\omega^{T}_{\varphi_{t}} = \omega^{T} + \sqrt{-1}\partial_{B}\bar \partial_{B} \varphi_{t}$. 
The functionals were defined by Futaki-Ono-Wang \cite{FOW}, and proved that the definition of the functions is independent of choice of the path $\{\varphi_{t}\ |\ t \in [a, b]\}$. 
Put $L_{\eta}(\varphi) := L(0, \varphi)$ and $M_{\eta}(\varphi) := M(0, \varphi)$. 
Then the critical points of $M_{\eta}$ give Sasaki-Einstein metrics which are compatible with the initial metric $g$. 
As in the K${\rm \ddot a}$hler geometry, the functionals $L$ and $M$ have the following properties. 
The proofs are given by the same arguments in K${\rm \ddot a}$hler geometry. 
A functional $H: \mathscr{H} \times \mathscr{H} \to \mathbb{R}$ is said to satisfy {\it the 1-cocycle condition} if 
\begin{enumerate}
\item $H(\varphi^{\prime}, \varphi^{\prime \prime}) + H(\varphi^{\prime \prime}, \varphi^{\prime}) = 0$, and 
\item $H(\varphi^{\prime}, \varphi^{\prime \prime}) + H(\varphi^{\prime \prime}, \varphi^{\prime \prime \prime}) + H(\varphi^{\prime \prime \prime}, \varphi^{\prime}) = 0$
\end{enumerate}
for each $\varphi^{\prime}, \varphi^{\prime \prime}, \varphi^{\prime \prime \prime} \in \mathscr{H}$. 
\begin{prop}\label{fundamental} \ 
\begin{enumerate}
\item The functionals $L$ and $M$ satisfy the 1-cocycle condition. 
\item $L(\varphi^{\prime}, \varphi^{\prime \prime} + C) = L(\varphi^{\prime} - C, \varphi^{\prime \prime}) = L(\varphi^{\prime}, \varphi^{\prime \prime}) + C$ ({\rm resp. }$M(\varphi^{\prime} + C, \varphi^{\prime \prime} + C) = M(\varphi^{\prime}, \varphi^{\prime \prime})$)
for each $\varphi^{\prime}, \varphi^{\prime \prime} \in \mathscr{H}$ and $C \in \mathbb{R}$. 
\end{enumerate}
\end{prop}
Hence we can define the mapping $M$ (denoted by the same $M$) on $\mathscr{S}(g)$ by 
\begin{equation*}
M(g^{\prime}, g^{\prime \prime}) := M(\varphi^{\prime}, \varphi^{\prime}), 
\end{equation*}
where $\varphi^{\prime}, \varphi^{\prime}$ are basic functions such that $g_{\varphi^{\prime}} = g^{\prime}$ and $g_{\varphi^{\prime}} = g^{\prime \prime}$. 
We call the functional $M_{\eta}$ on the space of Sasaki metrics which have the same basic K${\rm \ddot a}$hler class as the initial metric $g$ defined by 
\begin{equation*}
M_{\eta}(g^{\prime}) := M(g, g^{\prime}) 
\end{equation*}
{\it the transverse K-energy map} of the Sasaki manifold $(S, g)$. 
\subsection{The functionals $I_{\eta}$ and $J_{\eta}$}
For each $\varphi^{\prime}, \varphi^{\prime \prime} \in \mathscr{H}$, we put 
\begin{align*}
I(\varphi^{\prime}, \varphi^{\prime \prime}) &:= \frac{1}{V}\int_{S}(\varphi^{\prime \prime} - \varphi^{\prime})\left( (\frac{1}{2}d\eta_{\varphi^{\prime}})^{n} \wedge \eta_{\varphi^{\prime}} - (\frac{1}{2}d\eta_{\varphi^{\prime \prime}})^{n} \wedge \eta_{\varphi^{\prime \prime}}\right), \\
J(\varphi^{\prime}, \varphi^{\prime \prime}) &:= \frac{1}{V}\int_{a}^{b}dt \int_{S}\dot{\varphi_{t}}\left( (\frac{1}{2}d\eta_{\varphi^{\prime}})^{n} \wedge \eta_{\varphi^{\prime}} - (\frac{1}{2}d\eta_{\varphi_{t}})^{n} \wedge \eta_{\varphi_{t}}\right), 
\end{align*}
where $\{\varphi_{t}\ |\ t \in [a, b]\}$ is an arbitrary piecewise smooth path in $\mathscr{H}$ such that $\varphi_{a} = \varphi^{\prime}$ and $\varphi_{b} = \varphi^{\prime \prime}$. 
The following lemma is proved by direct calculations. 
\begin{prop}\label{IJ1}\
\begin{enumerate}
\item $J(\varphi^{\prime}, \varphi^{\prime \prime}) = \frac{1}{V}\int_{S}(\varphi^{\prime \prime} - \varphi^{\prime})(\frac{1}{2}d\eta_{\varphi^{\prime}})^{n} \wedge \eta_{\varphi^{\prime}} - L(\varphi^{\prime}, \varphi^{\prime \prime})$. 
In particular, the definition of $J$ is independent of choice of the path $\{\varphi_{t}\ |\ t \in [a, b]\}$. 
\item $I(\varphi^{\prime} + C, \varphi^{\prime \prime} + C) = I(\varphi^{\prime}, \varphi^{\prime \prime})$ and $J(\varphi^{\prime} + C, \varphi^{\prime \prime} + C) = J(\varphi^{\prime}, \varphi^{\prime \prime})$ for each $\varphi^{\prime}, \varphi^{\prime \prime} \in \mathscr{H}$ and constant $C \in \mathbb{R}$. 
\end{enumerate}
\end{prop}

By Proposition \ref{IJ1}, we can define the mappings $I$ and $J$ (denoted by the same notations) on $\mathscr{S}(g)$ by 
\begin{equation*}
I(g^{\prime}, g^{\prime \prime}) := I(\varphi^{\prime}, \varphi^{\prime \prime})\ \text{and}\ 
J(g^{\prime}, g^{\prime \prime}) := J(\varphi^{\prime}, \varphi^{\prime \prime}), 
\end{equation*}
where $\varphi^{\prime}, \varphi^{\prime}$ are basic functions such that $g_{\varphi^{\prime}} = g^{\prime}$ and $g_{\varphi^{\prime}} = g^{\prime \prime}$. 
Put 
\begin{equation*}
I_{\eta}(g^{\prime}) := I(g, g^{\prime}),\ 
J_{\eta}(g^{\prime}) := J(g, g^{\prime}) 
\end{equation*}
for each $g^{\prime} \in \mathscr{S}(g)$. 
The functional $J$ does not satisfy the $1$-cocycle condition in general, but it satisfies the following equality; 
\begin{align*}
&J(\varphi^{\prime}, \varphi^{\prime \prime}) + J(\varphi^{\prime \prime}, \varphi^{\prime \prime \prime}) \\
&\qquad=  J(\varphi^{\prime}, \varphi^{\prime \prime \prime})- \frac{1}{V}\int_{S}(\varphi^{\prime \prime \prime} - \varphi^{\prime \prime})\left( (\frac{1}{2}d\eta_{\varphi^{\prime}})^{n} \wedge \eta_{\varphi^{\prime}} - (\frac{1}{2}d\eta_{\varphi^{\prime \prime}})^{n} \wedge \eta_{\varphi^{\prime \prime}} \right). 
\end{align*}

Put $I_{\eta}(\varphi) := I(0, \varphi)$ and $J_{\eta}(\varphi) := J(0, \varphi)$ for each $\varphi \in \mathscr{H}$. 
We now take an arbitrary smooth path $\{ \varphi_{t}\ |\ t \in [a, b] \}$ in $\mathscr{H}$. 
Then by a simple calculation we have 
\begin{equation}\label{differential}
\frac{d}{dt}\left( I_{\eta}(\varphi_{t}) - J_{\eta}(\varphi_{t}) \right) 
= \frac{1}{V}\int_{S}\varphi_{t} \Box_{t} \dot{\varphi_{t}}(\frac{1}{2}d\eta_{\varphi_{t}})^{n} \wedge (\eta_{\varphi_{t}}), 
\end{equation}
where $\Box_{t} = \frac{1}{2}\Delta_{B, t}$ is the basic complex Laplacian with respect to the Sasakian metric $g_{\varphi_{t}}$. 
The following properties of $I_{\eta}$ and $J_{\eta}$ are essential to obtain the $C^{0}$-estimate for the solutions of equation (\ref{MA2}). 
\begin{prop}\label{functional_inequality}
$I_{\eta}, I_{\eta} - J_{\eta}, J_{\eta}$ are non negative functionals and satisfies the following inequality; 
\begin{equation*}
0 \leq I_{\eta}(\varphi) \leq (n + 1)(I_{\eta}(\varphi) - J_{\eta}(\varphi)) \leq n I_{\eta}(\varphi). 
\end{equation*}
\end{prop}
Propositions \ref{fundamental}, \ref{IJ1} and \ref{functional_inequality} can be obtained by a similar way as in K${\rm \ddot a}$hler cases (see \cite{Mabuchi2} for example). 


\begin{thebibliography}{99}
%
%
\bibitem{BG}
C. Boyer and K. Galicki, 
Sasakian geometry, 
Oxford mathematical monographs (2008) 
\bibitem{BM}
S. Bando and T. Mabuchi, Uniqueness of Einstein K${\rm \ddot a}$hler metrics modulo connected group actions, 
Algebraic geometry, Sendai, 1985, Adv. Stud. Pure Math., 10, 11-40 (1987) 
\bibitem{C}
X. X. Chen, 
The space of K${\rm \ddot a}$hler metrics, 
J. Differential Geometry, 56, 189--234 (2000) 
\bibitem{CFO}
K. Cho, A. Futaki and H. Ono, 
Uniqueness and examples of toric Sasaki-Einstein manifolds, 
Comm. Math. Phys., 277, 439-458 (2008) 
\bibitem{El}
A. El Kacimi-Alaoui, 
Op${\rm \grave{e}}$rateurs transversalement elliptiques sur un feuilletage riemannien
et applications, Compositio Math. 79, 57--106 (1990) 
\bibitem{F}
A. Futaki, 
K${\rm \ddot a}$hler-Einstein metrics and integral invariants, 
Lecture Notes in Mathematics, 1314. Springer-Verlag, Berlin (1988) 
\bibitem{FOW}
A. Futaki, H. Ono and G. Wang, 
Transverse K${\rm \ddot a}$hler geometry of Sasaki manifolds and 
toric Sasaki-Einstein manifolds, 
to appear in J. Differential Geometry, math.DG/0607586. 
\bibitem{G}
D. Guan, 
On modified Mabuchi functional and Mabuchi moduli space of K${\rm \ddot a}$hler metrics on toric bundles, 
Math. Res. Letters, 6, 547--555 (1999) 
\bibitem{HS}
I. Hasegawa and M. Seino, 
Some remarks on Sasakian geometry-applications of Myers' theorem and the canonical affine connection, J. Hokkaido Univ. Education 32 1-7, (1981) 
\bibitem{Li}
A. Lichnerowicz, G${\rm \acute e}$om${\rm \acute e}$tie des groupes de transformations, Dunod, Paris (1958)
\bibitem{Mabuchi}
T. Mabuchi, 
Some Symplectic geometry on compact K${\rm \ddot a}$hler manifolds. I, 
Osaka J. Math. 24, 227--252 (1987) 
\bibitem{Mabuchi2}
T. Mabuchi, 
K-energy maps integrating Futaki invariants, T${\rm \hat o}$hoku Math.
Journ. 38, 575--593, (1986) 
\bibitem{M}
R. Montgomery, 
A tour of subriemannian geometries, their geodesics and applications, 
Mathematical surveys and monographs, Volume 91 (2002) 
\bibitem{YN}
Y. Nitta, 
A diametter bound for Sasaki manifolds with application to uniqueness for Sasaki-Einstein structure, 
arXiv:0906.0170v2. 
\bibitem{RR}
M. Ritor${\rm \acute{e}}$ and C. Rosalesa, 
Area-stationary surfaces in the Heisenberg group $\mathbb{H}^{1}$, 
Advances in Mathematics, Volume 219, Issue 2, 633--671 (2008)
\bibitem{RS1}
R. S. Strichartz, 
Sub-Riemannian geometry, 
J. differential geometry, 24, 221--263 (1986) 
\bibitem{RS2}
R. S. Strichartz, 
Corrections to ``Sub-Riemannian geometry", 
J. differential geometry, 30, 595--596 (1989)
\bibitem{S}
Ken'ichi Sekiya, 
On the uniqueness of Sasaki-Einstein metrics, 
arXiv: 0906.2665v1. 
\bibitem{Yau}
S. T. Yau, 
On the Ricci curvature of a compact K${\rm \ddot a}$hler manifold and the complex Monge-Amp${\rm \grave e}$re equation I, 
Comm. Pure Appl. Math., 31, 339-441 (1978) 
\end{thebibliography}


\end{document}